\documentclass[10pt]{amsart}
\usepackage{amssymb}
\usepackage{bm}
\usepackage{graphicx}
\usepackage[centertags]{amsmath}
\usepackage{amsfonts}
\usepackage{amsthm}
\usepackage{graphicx}
\usepackage{xcolor}
\newtheorem{thm}{Theorem}
\newtheorem{cor}[thm]{Corollary}
\newtheorem{lem}[thm]{Lemma}

\newtheorem{prop}[thm]{Proposition}
\newtheorem{claim}[thm]{Claim}

\newtheorem{defn}[thm]{Definition}
\theoremstyle{definition}
\newtheorem{rem}{Remark}

\newtheorem{notation}{Notation}

\newcommand{\nn}{\mathbb{N}}

\newcommand{\con}{\smallfrown}
\newcommand{\meg}{\geqslant}
\newcommand{\mik}{\leqslant}

\newcommand{\pred}{\mathrm{Pred}}

\newcommand{\ws}{\mathrm{ws}}
\newcommand{\w}{\mathrm{W}}

\newcommand{\supp}{\mathrm{supp}}

\newcommand{\lex}{\mik_{\mathrm{lex}}}

\begin{document}
\title{A Dual Ramsey theorem for trees}
\author{Stevo Todorcevic}
\author{Konstantinos Tyros}

\address{Department of Mathematics, University of Toronto, Toronto, Canada, M5S 2E4. Institut de Math\'ematiques de Jussieu, UMR 7586, 2 pl. Jussieu, case 7012, 75251 Paris Cedex 05, France.
Matemati\v{c}ki Institut, SANU, Belgrade, Serbia}
\email{stevo@math.toronto.edu stevo.todorcevic@imj-prg.fr 
stevo.todorcevic@sanu.ac.rs}
\address{and}
\address{Department of Mathematics, University of Athens, Panepistimiopolis 157 84, Athens, Greece}
\email{ktyros@math.uoa.gr}


\maketitle

\begin{abstract}
  We prove a dualization of the Graham--Rothschild Theorem for variable words indexed by homogeneous trees.
\end{abstract}

\section{Introduction} In 1971, R. Graham and B. Rothschild (see \cite{GR}) proved a higher dimensional analogue of the celebrated Hales-Jewett Theorem (see \cite{HJ}) about colorings of parameter words indexed by $\{0,...,n-1\}$ where $n$ is a positive integer number. To state the  Graham--Rothschild Theorem we need to introduce some pieces of notation.
Let $\Lambda$ be a finite nonempty set, which we view as an alphabet, let $m,n$ be positive integers with $m \mik n$
and let $v_0,...,v_{m-1}$ be symbols not belonging to $\Lambda$.
We view the elements of the set  $\Lambda^n$ of all functions from the set $\{0,...,n-1\}$ into $\Lambda$ as constant words.
By the term $m$-parameter word we mean a map
$w=w(v_0,...,v_{m-1}):\{0,...,n-1\}\to \Lambda\cup\{v_0,...,v_{m-1}\}$ such that every $v_i$ is attained at least once
and for every integer with $0\mik i<m-2$ we have $\min w^{-1}(\{v_i\})< \min w^{-1}(\{v_{i+1}\})$. Given an $m$-parameter word $w=w(v_0,...,v_{m-1})$ and $a_0,...,a_{m-1}$ in $\Lambda$ we denote by $w(a_0,...,a_{m-1})$ the element of $\Lambda^n$
resulting by substituting every occurrence of each $v_i$ by $a_i$ and set $[w]_\Lambda = \{w(a_0,...,a_{m-1}):a_0,...,a_{m-1}\in\Lambda\}$. Moreover, if $k$ is a positive integer with $k\mik m$ then we denote by $[w]_{k,\Lambda}$ the set of all $k$-parameter words $\tilde{w}$ satisfying $[\tilde{w}]_{\Lambda}\subseteq [w]_{\Lambda}$.
The Graham--Rothschild Theorem states the following.
\begin{thm}[Graham--Rothschild]
\label{Graham_Rothschild}
  Let $\ell,k,m$ and $r$ be positive integers with $k\mik m$. Then there exists a positive integer $n_0$ with the following property. For every finite alphabet $\Lambda$ with $\ell$ elements, every integer $n$ with $n\meg n_0$ and every $r$-coloring of the $k$-parameter words, there exists an $m$-parameter word $w$ such that the set $[w]_{k,\Lambda}$ is monochromatic.
\end{thm}

One of the immediate consequences of the Graham--Rothschild Theorem is a finite form of the Milliken--Taylor Theorem
(see \cite{Mi}, \cite{Ta}).
Let $k,m$ be positive integers with $k\mik m$ and let $\mathbf{F} = (F_i)_{i=0}^{m-1}$ be a sequence of pairwise
disjoint finite nonempty subsets of the natural numbers such that $(\min F_i)_{i=0}^{m-1}$ is an increasing sequence.
By $\mathrm{DS}_k(\mathbf{F})$ we denote the set of all pairwise disjoint sequences $(G_j)_{j=0}^{k-1}$ of nonempty sets such that $(\min G_j)_{j=0}^{m-1}$ is an increasing sequence and each $G_j$ belongs to the set $\{\bigcup_{i\in U}F_i:U\subseteq\{0,...,m-1\}\}$. For simplicity, if $n$ is a positive integer, then we set
$\mathrm{DS}_k(n) = \mathrm{DS}_k((\{i\})_{i=0}^{n-1})$. The finite form of the Milliken--Taylor Theorem is as follows.

\begin{thm}(Milliken--Taylor)
  Let $k,m$ and $r$ be positive integers. Then there exists a positive integer $n_0$ with the following property.
  For every positive integer $n$ with $n_0\mik n$ and every $r$-coloring of $\mathrm{DS}_k(n)$ there exists
  $\mathbf{F}$ in $\mathrm{DS}_m(n)$ such that the set $\mathrm{DS}_k(\mathbf{F})$ is monochromatic.
\end{thm}

Notice that the case $m=1$ of the above result boils down to Folkman's Theorem (see \cite{GRS}; p.82),
which is a finite form of Hindman's Theorem (see \cite{H}). The main result of the present work is a
Graham--Rothschild Theorem for variable words indexed by homogeneous trees.
To state our result we need to introduce some pieces of notation.

\subsection{Complete Skew trees}
Let $b$ and $n$ be positive integers. We denote by $b^{<n}$ the set of all sequences in $\{0,...,b-1\}$ of length less than $n$. For every element $s$ in $b^{<n}$ we denote by $|s|$ its length, while for every $i$ in $\{0,...,|s|-1\}$ we denote by $s(i)$ the $i$-th element of $s$. 
%
By $\sqsubseteq$ we denote the initial segment ordering and by $\lex$ we denote the
lexicographical order on the set $b^{<n}$.

We view the subsets of $b^{<n}$ as subtrees.
For every subtree $S$ of $b^{<n}$ and $s$ in $S$, we define the following. We set $\mathrm{Pred}_S(s)=\{t\in S:\;t\sqsubset s\}$ and we denote by $\mathrm{ImSucc}_S(s)$ the set of all $t$ in $S$ such that $s\sqsubset t$ and there is no $t'$ in $S$ satisfying $s\sqsubset t'\sqsubset t$. Finally, we set $\mathrm{h}_S(s)$ to be the cardinality of the set $\mathrm{Pred}_S(s)$.

Let $k$ be a positive integer. A $k$-complete skew subtree $S$ of $b^{<n}$ is a subtree with the following properties.
\begin{enumerate}
  \item[(i)] $S$ has a minimum with respect to $\sqsubseteq$ and every maximal chain with respect to $\sqsubseteq$
  is of cardinality $k$.
  \item[(ii)] For every $s$ and $s'$ in $S$ with $\mathrm{h}_S(s)=\mathrm{h}_S(s')$ and $s\lex s'$, we have that $|s|\mik|s'|$.
  \item[(iii)] For every $s$ and $s'$ in $S$ with $\mathrm{h}_S(s)<\mathrm{h}_S(s')$, we have that $|s|<|s'|$.
  \item[(iv)] For every non maximal $s$ in $S$ with respect to $\sqsubseteq$ and every $i$ in $\{0,...,b-1\}$,
                there is unique $t$ in $\mathrm{ImSucc}_S(s)$ such that $s^\con (i)\sqsubseteq t$.
\end{enumerate}
\subsection{Variable words indexed by trees}.
Let $b, n$ be positive integers and let $\Lambda$ be a finite alphabet. By $\w(b,n,\Lambda)$ we denote the set of all maps
from $b^{<n}$ to $\Lambda$.
Let $k$ be a positive integer and $T$ a $k$-complete skew subtree of $b^{<n}$. A $T$-variable word $f$ on $\w(b,n,\Lambda)$ is a map from $b^{<n}$ into $\Lambda\cup\{v_t:t\in T\}$ such that for every $t$ in $T$ the set $f^{-1}(v_t)$ is non empty and has $t$ as a $\sqsubseteq$-minimum. We denote by $\w_{v,T}(b,n,\Lambda)$ the set of all $T$-variable words on $\w(b,n,\Lambda)$ and by $\w_{v,k}(b,n,\Lambda)$
the union of $\w_{v,T}(b,n,\Lambda)$ over all possible $k$-complete skew subtrees $T$ of $b^{<n}$.
Observe that for every $f$ in $\w_{v,k}(b,n,\Lambda)$, there exists unique $k$-complete skew subtree $T$ of $b^{<n}$,
which we denote by $\ws(f)$ such that $f$ is a $T$-variable word.

Let $f$ in $\w_{v,k}(b,n,\Lambda)$.
If $(\alpha_t)_{t\in \ws(f)}\in \Lambda^{\ws(f)}$, then we denote by $f\big((\alpha_t)_{t\in \ws(f)}\big)$ the element of
$\w(b,n,\Lambda)$ resulting by substituting each occurrence of $v_t$ in $f$ by $\alpha_t$ for all $t$ in $\ws(f)$.
Moreover, we define the span of $f$ as
\[[f]_\Lambda=\{f\big((\alpha_t)_{t\in \ws(f)}\big): (\alpha_t)_{t\in \ws(f)}\in \Lambda^{\ws(f)}\}.\]
Finally, if $k'$ is a positive integer with $k'\mik k$, then we set
\[\w_{v,k'}(f) = \{g\in \w_{v,k'}(b,n,\Lambda) : [g]_\Lambda\subseteq[f]_\Lambda\}.\]

\subsection{Main result} The main result of the present work states the following.
\begin{thm}
  \label{tree_GR}
  For every choice of positive integers $b,\ell, k,m$ and $r$ there exists a positive integer $n_0$ with the following property. For every integer $n$ with $n\meg n_0$, every finite alphabet $\Lambda$ with $\ell$ elements and every $r$-coloring of $\w_{v,k}(b,n,\Lambda)$ there exists $f$ in $\w_{v,m}(b,n,\Lambda)$ such that the set $\w_{v,k}(f)$ is monochromatic.
  We denote the least such $n_0$ by $\mathrm{TGR}(k,m,b,\ell,r)$.
\end{thm}
Notice that the case $b=1$ of Theorem \ref{tree_GR} yields Theorem \ref{Graham_Rothschild}.
The proof of Theorem \ref{tree_GR} requires the development of a Hales--Jewett type theorem for trees
(see Theorem \ref{tree_HJ}), which is of its own interest.
The proof of the latter result utilizes a double induction scheme resulting non primitive recursive bounds.

\subsection{A Milliken--Taylor theorem for trees} As the Graham--Rothschild Theorem yields the Milliken-Taylor Theorem,
it is expected that Theorem \ref{tree_GR} yields some form of the Milliken-Taylor Theorem for trees.
To state the latter we need some pieces of notation.
Let $b$ and $n$ be positive integers. We set
\[\mathcal{U}_1(b^{<n})=\{U\subseteq b^{<n}:\;U\;\text{has a minimum}\}.\]
The subspaces under consideration are collections of pairwise disjoint elements of $\mathcal{U}_1(b^{<n})$ such that their minimums form a complete skew subtree. More precisely, if $k$ is a positive integer, by $\mathcal{U}_k(b^{<n})$ we denote the set of all collections $(U_t)_{t\in T}$, where $T$ is a $k$-complete skew subtree of $b^{<n}$ and $U_t$ belongs to $\mathcal{U}_1(b^{<n})$ with $\min U_t=t$ for all $t$ in $T$. We denote by $\mathcal{U}(b^{<n})$ the union of $\mathcal{U}_k(b^{<n})$ over all positive integers $k$. For $\mathbf{U}$ and $\mathbf{V}$ in $\mathcal{U}(b^{<n})$, we say that $\mathbf{U}$ is a subspace of
$\mathbf{V}$, we write $\mathbf{U}\leq\mathbf{V}$, if every member of $\mathbf{U}$ is a union of members of $\mathbf{V}$. We denote by $\mathcal{U}(\mathbf{U})$ the set of all subspaces of $\mathbf{U}$, while if $k$ is a positive integer, we set $\mathcal{U}_k(\mathbf{U})=\mathcal{U}(\mathbf{U})\cap\mathcal{U}_k(b^{<n})$.
Theorem \ref{tree_GR} has the following immediate consequence.

\begin{thm}
  \label{subsets}
  Let $k,m,b,r$ be positive integers with $k\mik m$. Then there exists a positive integer $n_0$ with the following property. For every integer $n$ with $n\meg n_0$ and every coloring of $\mathcal{U}_k(b^{<n})$ with $r$ colors, there exists $\mathbf{U}$ in $\mathcal{U}_m(b^{<n})$ such that the set $\mathcal{U}_k(\mathbf{U})$ is monochromatic.
\end{thm}
Theorem \ref{subsets} is the result announced in \cite{TT} and can be viewed as a finite and multidimensional version of the main result in \cite{TT}.

\section{Background Material}
One of the main tools that we will need in the present work is a variant of the Hales--Jewett Theorem, that is, Theorem \ref{Hales_Jewett*} below. Although its proof is an straightforward modification of Shelah's proof of the
Hales--Jewett Theorem, we include its proof for the sake of completeness in an appendix.
Let us start with some basic notation that we will make use for the rest of the paper.
By $\nn$ we denote the set of the natural numbers starting from 0.
For every two sets $X$ and $Y$ by $Y^X$ we denote the set of all maps from $X$ into $Y$.
For every choice of sets $X,Y,Z$ with $Z\subseteq X$ and $f$ in $Y^X$ we denote by $f\upharpoonright Z$ the restriction of $f$ to $Z$.
For every two sequences of finite length $\mathbf{a}=(a_i)_{i=0}^{n-1}$ and $\mathbf{b} = (b_i)_{i=0}^{m-1}$ we denote their concatenation by
$\mathbf{a}^\smallfrown\mathbf{b}$, that is, the sequence $(z_i)_{i=0}^{n+m-1}$ satisfying $z_i = a_i$ for every $i=0,...,n-1$ and $z_i = b_i$ for all $i=n,...,n+m-1$.

\subsection{The Hales--Jewett Theorem}
One of the most central results in Combinatorics is the Hales--Jewett Theorem \cite{HJ}. To state it we need some notation.
Let $N$ be a positive integer and $\Lambda$ a finite alphabet. We view the elements of the set  $\Lambda^N$ of all functions from the set $\{0,...,N-1\}$ into $\Lambda$ as constant words of length $N$ over the alphabet $\Lambda$. Also, let a symbol $v\not\in\Lambda$. A variable word $w(v)$ of length $N$ over $\Lambda$ is a function from $\{0,...,N-1\}$ into $\Lambda\cup\{v\}$ such that the symbol $v$ occurs at least once. For a variable word $w(v)$ over $\Lambda$ and a letter $a\in\Lambda$ we denote by $w(a)$ the constant word resulting by substituting each occurrence of $v$ by $a$.
A combinatorial line is a set of the form $\{w(a):\;a\in\Lambda\}$, where $w(v)$ is a variable word over $\Lambda$.

\begin{thm}
  [Hales--Jewett]
  \label{Hales_Jewett}
  Let $k$ and $r$ be positive integers. Then there exists a positive integer $N_0$ with the following property.
  For every positive integer $N$ with $N\meg N_0$, every alphabet $\Lambda$ with $k$ elements and every $r$-coloring of $\Lambda^N$, there exists a monochromatic combinatorial line. We denote the least such  $N_0$ by $\mathrm{HJ}(k,r)$.

    Moreover, the numbers $\mathrm{HJ}(k,r)$ are upper bounded by a primitive recursive function belonging to the class $\mathcal{E}^5$ of Grzegorczyk's hierarchy.
\end{thm}

The bounds mentioned in the above theorem are due to Shelah \cite{Sh}.
We will need a multidimensional version of this Theorem. Let $N$ be a positive integer and $\Lambda$ be a finite alphabet. Also let $m$ be a positive integer and $v_0,...,v_{m-1}$ distinct symbols not belonging to $\Lambda$. An $m$-dimensional variable word of length $N$ over $\Lambda$ is a function $w:\{0,...,N-1\}\to\Lambda\cup\{v_0,...,v_{m-1}\}$ such that each $v_i$ occurs and their supports are in block position, i.e. $\max\supp_w(v_i)<\min\supp_w(v_{i+1})$ for all $i=0,...,m-2$, where $\supp_w(v_i)=\big\{j\in\{0,...,N-1\}:w(j)=v_i\big\}$ which we will refer to as the wildcard set of $v_i$ in $w$. In order to highlight the role of the variables, we will denote such an $m$-dimensional variable word by $w(v_0,...,v_{m-1})$. Thus a variable word is just an $1$-dimensional variable word.

Similarly to the case of a single variable word, we consider substitutions also for $m$-variable
words. In particular, if $a_0,...,a_{m-1}$ are elements of $\Lambda$, then by $w(a_0,...,a_{m-1})$ we denote the constant word resulting by substituting each occurrence of $v_i$ by $a_i$ for all $i=0,...,m-1$. An $m$-dimensional combinatorial subspace $S$ of $\Lambda^N$, is a set of the form $\{w(a_0,...,a_{m-1}):a_0,...,a_{m-1}\in\Lambda\}$, where $w(v_0,...,v_{m-1})$ is an $m$-dimensional variable word of length $N$ over $\Lambda$. In this case, we say that $w(v_0,...,v_{m-1})$ generates $S$. Let us observe that for every $m$-dimensional combinatorial subspace $S$ of $\Lambda^N$ there is unique $m$-dimensional word of length $N$ over $\Lambda$ generating $S$. The multidimensional version of Theorem \ref{Hales_Jewett} is the following.
\begin{thm}
  [Hales--Jewett]
  \label{Hales_Jewett_mult_simple}
  Let $k,m,r$ be positive integers. Then there exists a positive integer $N_0$ with the following property.
  For every positive integer $N$ with $N\meg N_0$, every alphabet $\Lambda$ with $k$ elements and every $r$-coloring of $\Lambda^N$, there exists a monochromatic $m$-dimensional combinatorial subspace. We denote the least such  $N_0$ by $\mathrm{MHJ}(k,m,r)$.

    Moreover, the numbers $\mathrm{MHJ}(k,m,r)$ are upper bounded by a primitive recursive function belonging to the class $\mathcal{E}^5$ of Grzegorczyk's hierarchy.
\end{thm}

However, we will need a mild strengthening of Theorem \ref{Hales_Jewett_mult_simple}.
To introduce this strengthening, let us denote for every subset $A$ of the natural numbers and every positive integer $\kappa$ the set of all subsets of $A$ with exactly $\kappa$ elements by $A^{(\kappa)}$.
In our case, the underling structure is the Cartesian product $\Lambda^N\times\{0,...,N-1\}^{(\kappa)}$.
First, for every $m$-dimensional variable word $w$ over $\Lambda$ of length $N$, let us set $\ell^w_i=\min\supp_w(v_i)$ for all $i=0,...,m-1$.
An $m$-dimensional $\kappa^*$combinatorial subspace is a set of the form $X=S\times\{\ell^w_i:i=0,...,m-1\}^{(\kappa)}$, where $w(v_1,...,v_{m-1})$ is an $m$-dimensional variable word over $\Lambda$ of length $N$ and $S$ the combinatorial subspace generated by $w$. In this case, we say that $w(v_0,...,v_{m-1})$ generates the $\kappa^*$combinatorial subspace $X$. Also, let us observe that for every $m$-dimensional $\kappa^*$combinatorial subspace $X$ there is unique $m$-dimensional variable word generating $X$. We consider the following special type of colorings.

\begin{defn}
  Let $\kappa,m,N$ be positive integers with $\kappa\mik m\mik N$ and let $\Lambda$ be a finite alphabet. Also let $w(v_0,...,v_{m-1})$ be an $m$-dimensional variable word over $\Lambda$ of length $N$ and let $X$ be the $\kappa^*$combinatorial subspace generated by $w$. We say that a finite coloring $c$ of $\Lambda^N\times\{0,...,N-1\}^{(\kappa)}$ is strongly $^*$insensitive in $X$ if
  \[c(w(a_0,...,a_{m-1}),\{\ell_i^w:i\in F\})=c(w(a_0',...,a_{m-1}'),\{\ell_i^w:i\in F\})\]
  for every choice of $F$ in $\{0,...,m-1\}^{(\kappa)}$ and $a_0,...,a_{m-1},a_0',...,a_{m-1}'$ in $\Lambda$ with $a_i=a_i'$ for all $i$ in $F$.
\end{defn}

It is clear that it is not expected the existence of a monochromatic $\kappa^*$combinatorial subspace given an arbitrary coloring. However, we have the following.

\begin{thm}
  \label{Hales_Jewett*}
  Let $k,\kappa,m,r$ be positive integers with $\kappa\mik m$. Then there exist integers $n_0,q_0,...,q_m$ with $0=q_0<q_1<...<q_m=n_0$ satisfying the following property. For every integer $N$ with $N\meg n_0$, every finite alphabet $\Lambda$ with $k$ elements and every $r$-coloring $c$ of $\Lambda^N\times\{0,...,N-1\}^{(\kappa)}$, there exists an $m$-dimensional $\kappa^*$combinatorial subspace $X$ such that $c$ is strongly $^*$insensitive in $X$.

  Moreover, if $w(v_0,...,v_{m-1})$ is the $m$-dimensional variable word generating $X$, then the wildcard set of $v_i$ in $w$ is contained in the interval $[q_i,q_{i+1})$ for all $i=0,...,m-1$. We denote the least such $n_0$ by $\mathrm{MHJ}^*(k,\kappa,m,r)$.

  Finally, the numbers $\mathrm{MHJ}^*(k,\kappa,m,r)$ are upper bounded by a primitive recursive function belonging to the class $\mathcal{E}^5$ of Grzegorczyk's hierarchy.
\end{thm}

The proof of Theorem \ref{Hales_Jewett*} is a straightforward modification of Shelah's proof \cite{Sh} of the Hales--Jewett Theorem. For the sake of completeness we include its proof in an appendix.
Theorem \ref{Hales_Jewett*} has the following easy consequence which is useful for us below.
\begin{cor}
  \label{Hales_Jewett_cor_v2}
  Let $k,\kappa,m,r$ be positive integers with $\kappa\mik m$. Then there exists a positive integer $q$ with the following property. For every integer $N$ with $N\meg mq$, every alphabet finite $\Lambda$ with at most $k$ elements, every collection $(G_j)_{j=0}^{\kappa-1}$ of non empty and pairwise disjoint subsets of $\{0,...,m-1\}$ and every $r$-coloring of $\Lambda^N\times\prod_{j=0}^{\kappa-1}\bigcup_{i\in G_j}I_i$, where $I_i=\{iq,...,(i+1)q-1\}$ for all $i=0,...,m-1$, there exists an $m$-dimensional variable word $w(v_0,...,v_{m-1})$ satisfying the following.
  \begin{enumerate}
    \item[(a)] The wildcard set of $v_i$ in $w$ is contained in $I_i$ for all $i=0,...,m-1$.
    \item[(b)] For every choice of $\alpha_0,...,\alpha_{m-1},\alpha'_0,...,\alpha'_{m-1}$ in $\Lambda$ and $(p_0,...,p_{\kappa-1})$ in $\prod_{i=0}^{\kappa-1}G_i$ with $\alpha_{p_i}=\alpha'_{p_i}$ for all $i=0,...,\kappa-1$, we have that
        \[c(w(\alpha_0,...,\alpha_{m-1}),(\ell^w_{p_i})_{i=0}^{\kappa-1})=c(w(\alpha'_0,...,\alpha'_{m-1}),(\ell^w_{p_i})_{i=0}^{\kappa-1}).\]
  \end{enumerate}
  We denote the least such $q$ by $\mathrm{Q}(k,\kappa,m,r)$.

  Finally, the numbers $\mathrm{Q}(k,\kappa,m,r)$ are upper bounded by a primitive recursive function belonging to the class $\mathcal{E}^5$ of Grzegorczyk's hierarchy.
\end{cor}
As mentioned above, Corollary \ref{Hales_Jewett_cor_v2} is an easy consequence of Theorem \ref{Hales_Jewett*} and in particular we have that $\mathrm{Q}(k,\kappa,m,r)\mik\mathrm{MHJ}^*(k,\kappa,m,r)$ for every choice of $k,\kappa,m$ and $r$ as in the relevant statements.


\section{Skew, complete skew and vector complete skew subtrees}
In this section we gather all the tree related notions and notation.

\subsection{Tree notation}
Let us fix positive integers $b$ and $n$. We denote by $b^{<n}$ (resp. $b^{\mik n}$) the set of all sequences in $\{0,...,b-1\}$ of length less than (resp. at most) $n$. We also denote by $\emptyset$ the empty sequence, as well as, the empty set. For every element $s$ in $b^{<n}$ we denote by $|s|$ its length, while for every $i$ in $\{0,...,|s|-1\}$ we denote by $s(i)$ the $i$-th element of $s$ and for every $i$ in $\{0,...,|s|\}$ we denote by $s\upharpoonright i$ the initial segment of $s$ of length $i$. In particular, $s = ( s(0),..., s(|s|-1) )$ for each non empty $s$ in $b^{<n}$. Moreover, for every $s,t$ in $b^{<n}$, we denote by $s\wedge t$ their longest common initial segment.

Let us observe that $b^{<n}$ with the initial segment ordering $\sqsubseteq$ is a uniquely rooted and balanced tree, that is, all maximal chains have the same length, with fixed branching number $b$, i.e. every non-maximal node has exactly $b$ immediate successors.
We will also consider the lexicographical order on the set $b^{<n}$ defined as follows. For every $s,t$ in $b^{<n}$, we set $s\lex t$ if either
\begin{enumerate}
  \item[(i)] we have that $s\sqsubseteq t$, or
  \item[(ii)] we have that $s\not\sqsubseteq t$ and $t\not\sqsubseteq s$ and $s(i)<t(i)$, where $i=|s\wedge t|$.
\end{enumerate}
Finally, we consider one more linear order $\preccurlyeq$ on $b^{<n}$ defined as follows. For every $s,t$ in $b^{<n}$, we set $s\preccurlyeq t$ if either
\begin{enumerate}
  \item[(i)] $|s|< |t|$, or
  \item[(ii)] $|s|=|t|$ and $t\lex s$.
\end{enumerate}

A subtree $S$ of $b^{<n}$ is a subset of $b^{<n}$ endowed with $\sqsubseteq$.
We say that a subtree $S$ of $b^{<n}$ is a subtree of some subtree $T$ of $b^{<n}$, if $S$ is a subset of $T$.
For every subtree $S$ of $b^{<n}$ and $s$ in $S$, we define the following. The set of the predecessors of $s$ inside $S$ is defined as $\mathrm{Pred}_S(s)=\{t\in S:\;t\sqsubset s\}$ and the set of the successors of $s$ inside $S$ is defined as $\mathrm{Succ}_{S}(s)=\{t\in S:\;s\sqsubseteq t\}$. Moreover, we consider the set of the immediate successors $\mathrm{ImSucc}_S(s)$ of $s$ in $S$ to be the set of all $t$ in $S$ such that $s\sqsubset t$ and there is no $t'$ in $S$ satisfying $s\sqsubset t'\sqsubset t$. The height of $s$ in $S$, denoted by $\mathrm{h}_S(s)$, is the cardinality of the set $\mathrm{Pred}_S(s)$, while the height of $S$, denoted by $\mathrm{h}_S$, is the maximal length of a chain in $S$. If $m$ is a non negative integer with $m<\mathrm{h}_S$, we define the $m$-th level of $S$, denoted by $S(m)$ to be the set of all $t$ in $S$ such that $\mathrm{h}_S(t)=m$. Finally, the level set of $S$, denoted by $L(S)$, is the set $\{m<n:b^m\cap S\neq\emptyset\}$.

\subsubsection{Skew subtrees}
A skew subtree $S$ of $b^{<n}$ is a subtree that satisfies the following.
\begin{enumerate}
  \item[(i)] $S$ has a minimum, called the root of $S$ and denoted by $\mathrm{r}_S$.
  \item[(ii)] For every two nodes $s,s'$ in $S$ with $\mathrm{h}_S(s)=\mathrm{h}_S(s')$ and $s\lex s'$, we have that $|s|\mik|s'|$.
  \item[(iii)] For every two nodes $s,s'$ in $S$ with $\mathrm{h}_S(s)<\mathrm{h}_S(s')$, we have that $|s|<|s'|$.
  \item[(iv)] There exist $i_S$ in $\{0,...,b-1\}$ and $s_S$ in $S$ satisfying the following.
  \begin{enumerate}
    \item[(a)] For every $s'$ in $S$ with $s'\preccurlyeq s_S$ and $s'\neq s_S$, we have that for every $i$ in $\{0,...,b-1\}$, there is unique $t$ in $\mathrm{ImSucc}_S(s')$ such that $s'^\con (i)\sqsubseteq t$. Thus, in particular, the set $\mathrm{ImSucc}_S(s')$ is of cardinality $b$.
    \item[(b)] For every $s'$ in $S$ with $s_S\preccurlyeq s'$ and $s'\neq s_S$, we have that the set $\mathrm{ImSucc}_S(s')$ is empty
    \item[(c)] The set $\mathrm{ImSucc}_S(s_S)$ is of cardinality $i_S+1$ and for every
        $i$ in $\{0,...,i_s\}$, there is unique $t$ in $\mathrm{ImSucc}_S(s_S)$ such that $s_S^\con (i)\sqsubseteq t$.
  \end{enumerate}
\end{enumerate}
Let us observe that by the condition (iv) of the definition of a skew subtree, we have that every skew subtree has at least two nodes. We will also view singletons as skew subtrees.

\subsubsection{Complete skew subtrees}
We will say that a skew subtree of $b^{<n}$  is $k$-complete, where $k$ is a positive integer, if every maximal chain of itself is of length $k$ and every non maximal node has exactly $b$ immediate successors in the subtree. Moreover, we call a skew subtree complete, if it is $k$-complete for some $k$.
Let us also observe that if $S$ is a $k$-complete skew subtree, then $\mathrm{h}_S = k$ and there exists unique isomorphism $\mathrm{I}_S:b^{<k}\to S$, that is, $\mathrm{I}_S$ is bijective and preserves $\sqsubseteq, \lex$ and $\preccurlyeq$.

\subsection{Vector tree notation}
Let us fix positive integers $b,n,d$. We define $\mathbf{T}_{b,n}^d$ to be the $d$-tuple having each component equal to $b^{<n}$. We view $\mathbf{T}_{b,n}^d$ as a vector tree. A vector subtree of $\mathbf{T}_{b,n}^d$ is a $d$-tuple such that each component is a subtree of $b^{<n}$. For $\mathbf{S}_1, \mathbf{S}_2$ vector subtrees of $\mathbf{T}_{b,n}^d$, we say that $\mathbf{S}_1$ is a vector subtree of $\mathbf{S}_2$, if every component of $\mathbf{S}_1$ is a subtree of the corresponding component of  $\mathbf{S}_2$.

\subsubsection{Vector skew subtrees}
Let $k$ be a positive integer. A vector subtree $\mathbf{S}=(S_1,...,S_d)$ of $\mathbf{T}_{b,n}^d$ is called a vector skew subtree of $\mathbf{T}_{b,n}^d$ of height $k$, if each $S_i$ is a skew subtree of $b^{<n}$ and the following are satisfied.
\begin{enumerate}
  \item[(i)] There exists $j_0$ in $[d]$ such that $S_i$ is $k$-complete skew subtree of $b^{<n}$ for all $i$ in $[d]$ with $i<j_0$ and $S_i$ is $(k-1)$-complete skew subtree of $b^{<n}$ for all $i$ in $[d]$ with $i>j_0$.
      Moreover, $S_{j_0}$ is of height $k$.
  \item[(ii)] For every triple of integers $i,j,m$ with $1\mik i<j\mik d$ and $0\mik m<k$, we have that $|s_i|\mik |s_j|$ for each $s_i,s_j$ satisfying $s_i\in S_i(m)$ and $s_j\in S_j(m)$.
  \item[(iii)] For every integer $m$ with $0\mik m<k-1$, we have that $|s|<|s'|$ for each $s,s'$ satisfying $s\in \bigcup_{i=1}^dS_i(m)$ and $s'\in \bigcup_{i=1}^dS_i(m+1)$.
\end{enumerate}
\subsubsection{Vector complete skew subtrees}
A vector skew subtree $\mathbf{S}=(S_1,...,S_d)$ of $\mathbf{T}_{b,n}^d$ is called vector $k$-complete if
each $S_i$ is a $k$-complete skew subtree of $b^{<n}$.
A vector subtree is called vector complete skew subtree of $\mathbf{T}_{b,n}^d$,
if it is a vector $k$-complete skew subtree of $\mathbf{T}_{b,n}^d$ for some positive integer $k$.
If $\mathbf{S}=(S_1,...,S_d)$ is either a vector strong subtree or a vector skew subtree of $\mathbf{T}^d_{b,n}$, then we set $\mathrm{h}_\mathbf{S}=\mathrm{h}_{S_1}$. Finally, for every positive integer $k$ and every $\mathbf{S}$ vector complete skew subtree of $\mathbf{T}_{b,n}^d$, we denote by $\mathcal{CT}_k(\mathbf{S})$ the set of all vector $k$-complete skew subtrees of $\mathbf{S}$.

\begin{rem}
  An important property of the vector complete skew subtrees is the following. Fix positive integers $b,n$ and pick a non-maximal node $s$ of $b^{<n}$. Let $T=\{t\in b^{<n}:t \preccurlyeq s\}$ and $(t_i)_{i=1}^d$ the $\sqsubseteq$-minimal elements of $b^{<n}\setminus T$ written in $\preccurlyeq$-increasing order. Also let $m,k$ be positive integers with $k\mik m\mik n-|s|$ and $(S_i)_{i=1}^d$ a vector $k$-complete skew subtree of $\mathbf{T}_{b,m}^d$. Then the set
  \[S=T\cup\bigcup_{i=1}^d\{t_i^\con s:s\in S_i\}\]
  forms a skew subtree of $b^{<n}$. Moreover, if $T'$ is a skew subtree of $T$ then $S$ contains a complete skew subtree $S'$ of itself of height $\mathrm{h}_{T'}+k-1$ such that $S'\cap T=T'$.
\end{rem}

We will also make use of the following Ramsey type result for complete skew subtrees. It follows easily using Milliken's theorem for trees (see \cite{Mi2}) and the notion of envelope (see \cite{To}). 

\begin{thm}
  \label{thm:RamseyCompleteSkewSubtrees}
  Let $b,k,m,d,r$ be positive integers such that $k\mik m$. Then there exists a positive integer $n_0$ with the following property. For every integer $n\meg n_0$ and every coloring of the set $\mathcal{CT}_k(\mathbf{T}^d_{b,n})$ with $r$-colors
  there exists $\mathbf{S}$ in $\mathcal{CT}_m(\mathbf{T}^d_{b,n})$ such that the set $\mathcal{CT}_k(\mathbf{S})$ is
  monochromatic. We denote the least such $n_0$ by $\mathrm{CT}(k,m,b,d,r)$.
\end{thm}

Actually, one can obtain primitive recursive bounds for the numbers $\mathrm{CT}(k,m,b,d,r)$ using the work of Sokić in \cite{Sok}.

\section{A Hales-Jewett type theorem for mixed products of words indexed by trees.}
\subsection{Spaces of Word and subspaces} Let $\Lambda$ be a finite alphabet. We consider three different types of spaces. Let $b,n$ be positive integers. We set $\w(b,n,\Lambda)=\Lambda^{(b^{<n})}$, that is the set of all maps from $b^{<n}$ into $\Lambda$. We also set
\[\w^\uparrow(b,n,\Lambda)=\w(b,n,\Lambda)\times b^n\;\text{and}\;\w^\bullet(b,n,\Lambda)=\w(b,n,\Lambda)\times b^{<n}.\]

In order to define variable words we reserve for every $s$ in $b^{<n}$ a distinct symbol $v_s$ not belonging to $\Lambda$ that we view as a variable.
Let $k$ be a positive integer and $T$ a $k$-complete skew subtree of $b^{<n}$. A $T$-variable word $f$ on $\w(b,n,\Lambda)$ is a map from $b^{<n}$ into $\Lambda\cup\{v_t:t\in T\}$ such that for every $t$ in $T$ the set $f^{-1}(v_t)$ is non empty and has $t$ as a $\sqsubseteq$-minimum. We denote by $\w_{v,T}(b,n,\Lambda)$ the set of all $T$-variable words on $\w(b,n,\Lambda)$.
For a positive integer $k$, we denote by $\w_{v,k}(b,n,\Lambda)$
the union of $\w_{v,T}(b,n,\Lambda)$ over all possible $k$-complete skew subtrees $T$ of $b^{<n}$. Finally, we denote by $\w_v(b,n,\Lambda)$ the union of $\w_{v,k}(b,n,\Lambda)$ over all positive integers $k$ with $k\mik n$.
We view the elements of $\w_v(b,n,\Lambda)$ as variable words.
Observe that for every $f$ in $\w_v(b,n,\Lambda)$, there exists unique $k$, which we denote by $\mathrm{h}(f)$, and unique $k$-complete skew subtree $T$ of $b^{<n}$, which we denote by $\ws(f)$ such that $f$ is a $T$-variable word.

As usual, for variable words we consider substitutions. Let $f$ in $\w_v(b,n,\Lambda)$.
If $(\alpha_t)_{t\in \ws(f)}\in \Lambda^{\ws(f)}$, then we denote by $f\big((\alpha_t)_{t\in \ws(f)}\big)$ the constant word resulting by substituting each occurrence of $v_t$ in $f$ by $\alpha_t$ for all $t$ in $\ws(f)$.
Moreover, we define the span of $f$ as
\[[f]_\Lambda=\{f\big((\alpha_t)_{t\in \ws(f)}\big): (\alpha_t)_{t\in \ws(f)}\in \Lambda^{\ws(f)}\}.\]
We view the span of a variable word as a combinatorial subspace of $\w(b,n,\Lambda)$.

The subspaces of $\w^\uparrow(b,n,\Lambda)$ are described by pairs of the form $(f,X)$, where $f$ belongs to  $\w_v(b,n,\Lambda)$ and $X$ is a subset of $b^n$ of cardinality $b^{\mathrm{h}(f)}$, such that $\ws(f)\cup X$ is a skew subtree of $b^{\mik n}$. We denote the set of all such pairs as $W_v^\uparrow(b,n,\Lambda)$. For $(f,X)$ in $W_v^\uparrow(b,n,\Lambda)$, we define its span as
\[[(f,X)]_\Lambda^\uparrow=[f]_\Lambda\times X.\]

The subspaces of $\w^\bullet(b,n,\Lambda)$ are described again by variable words. In particular, if $f$ belongs to $\w_v(b,n,\Lambda)$, then we define
\[[f]_\Lambda^\bullet=[f]_\Lambda\times\ws(f).\]

Actually, we deal with mixed products of the above spaces. More precisely, let us also fix some positive integer $d$ and a partition $\mathbf{D}=(D_0,D_1,D_2)$ of $[d]$. We define
\[\w^d(b,n,\Lambda)=\prod_{i\in [d]}\w(b,n,\Lambda)\]
and
\[\w^\mathbf{D}(b,n,\Lambda)=\Big(\prod_{i\in D_0}\w(b,n,\Lambda)\Big)\times
\Big(\prod_{i\in D_1}\w^\uparrow(b,n,\Lambda)\Big)\times\Big(\prod_{i\in D_2}\w^\bullet(b,n,\Lambda)\Big).\]
Observe that the set $\w^\mathbf{D}(b,n,\Lambda)$ can naturally be identified with the Cartesian product
$\w^d(b,n,\Lambda)\times(b^n)^{D_1}\times(b^{<n})^{D_2}$. More precisely this identification is achieved by corresponding to
each element $((w_i)_{i\in D_0}, (w_i,x_i)_{i\in D_1}, (w_i,t_i)_{i\in D_2} )$ of the set $\w^\mathbf{D}(b,n,\Lambda)$ the
triple $((w_i)_{i\in[d]},(x_i)_{i\in D_1},(t_i)_{i\in D_2})$. We keep in mind this identification and when we refer to
element of $\w^\mathbf{D}(b,n,\Lambda)$, we describe them as triples of the form
\[((w_i)_{i\in[d]},(x_i)_{i\in D_1},(t_i)_{i\in D_2}),\]
where $w_i$ belongs to $\w(b,n,\Lambda)$ for each $i$ in $[d]$, $x_i$ belongs to $b^n$ for each $i$ in $D_1$ and
$t_i$ belongs to $b^{<n}$ for each $i$ in $D_2$.

A variable word of $\w^d(b,n,\Lambda)$ is a $d$-tuple $(f_i)_{i=1}^d$ of variable words, i.e. of elements of $\w_v(b,n,\Lambda)$, such that $(\ws(f_i))_{i=1}^d$ forms a vector complete skew subtree of $\mathbf{T}_{b,n}^d$. The span of a variable word $(f_i)_{i=1}^d$ of $\w^d(b,n,\Lambda)$ is defined as
\[[(f_i)_{i=1}^d]_\Lambda=\prod_{i\in[d]}[f_i]_\Lambda.\]
A variable word of $\w^\mathbf{D}(b,n,\Lambda)$ is a pair $\big((f_i)_{i=1}^d,(X_i)_{i\in D_1}\big)$ such that $(f_i)_{i=1}^d$ is a variable word of $\w^d(b,n,\Lambda)$ and $(f_i,X_i)$ belongs to $\w_v^\uparrow(b,n,\Lambda)$ for all $i$ in $D_1$. Let us denote the set of all such pairs by $\w_v^\mathbf{D}(b,n,\Lambda)$. Finally, for $\big((f_i)_{i=1}^d,(X_i)_{i\in D_1}\big)$ in $\w_v^\mathbf{D}(b,n,\Lambda)$, we define its span as
\[\big[\big((f_i)_{i=1}^d,(X_i)_{i\in D_1}\big)\big]_\Lambda=
\Big(\prod_{i\in D_0}[f_i]_\Lambda\Big)\times
\Big(\prod_{i\in D_1}[(f_i,X_i)]^\uparrow_\Lambda\Big)\times\Big(\prod_{i\in D_2}[f_i]^\bullet_\Lambda\Big)\]
and again when we refer to elements of $\big[\big((f_i)_{i=1}^d,(X_i)_{i\in D_1}\big)\big]_\Lambda$, we describe them as triples of the form
\[((w_i)_{i\in[d]},(x_i)_{i\in D_1},(t_i)_{i\in D_2}),\]
where $w_i$ belongs to $[f_i]_\Lambda$ for each $i$ in $[d]$, $x_i$ belongs to $X_i$ for each $i$ in $D_1$ and
$t_i$ belongs to $\ws(f_i)$ for each $i$ in $D_2$.
We also view such spans as combinatorial subspaces of $\w^\mathbf{D}(b,n,\Lambda)$.
We say that $\big[\big((f_i)_{i=1}^d,(X_i)_{i\in D_1}\big)\big]_\Lambda$ is the combinatorial subspace generated by
$\big((f_i)_{i=1}^d,(X_i)_{i\in D_1}\big)$ and its dimension is $\mathrm{h}(f_1)$.

\subsection{Hales--Jewett for mixed products} The main result of this section is Theorem \ref{tree_HJ} below.
Roughly speaking, the theorem states that given any arbitrary finite coloring of $\w^\mathbf{D}(b,n,\Lambda)$,
there exists a combinatorial subspace of $\w^\mathbf{D}(b,n,\Lambda)$ such that the restriction of the color on
the subspace becomes ``canonical''. The next definition describes the appropriate notion of ``canonical''
color required in our case.

\begin{defn}
  Let $d,b,n$ be positive integers, let $\mathbf{D}=(D_0,D_1,D_2)$ be a partition of $[d]$ and let $\Lambda$ be a finite alphabet.
  Also, let $c$ be a finite coloring of $\w^\mathbf{D}(b,n,\Lambda)$ and let $((f_i)_{i\in[d]},(X_i)_{i\in D_1})$ be an element of $\w^\mathbf{D}_v(b,n,\Lambda)$.
  We say that the combinatorial subspace generated by $((f_i)_{i\in[d]},(X_i)_{i\in D_1})$ is $c$-good, if
  for every choice of elements $((w_i)_{i\in[d]}, (x_i)_{i\in D_1}, (t_i)_{i\in D_2})$ and $((w'_i)_{i\in[d]}, (x'_i)_{i\in D_1}, (t'_i)_{i\in D_2} )$ in the combinatorial subspace $\big[\big((f_i)_{i=1}^d,(X_i)_{i\in D_1}\big)\big]_\Lambda$ with $t_i=t'_i$ and $w_i(t_i)=w'_i(t'_i)$ for all $i$ in $D_2$, we have that
\[c(((w_i)_{i\in[d]}, (x_i)_{i\in D_1}, (t_i)_{i\in D_2} ))=
c(((w'_i)_{i\in[d]}, (x'_i)_{i\in D_1}, (t'_i)_{i\in D_2} )).\]
\end{defn}

\begin{defn}
  Let $d_0,d_1,d_2$ be non negative integers with $d_0+d_1+d_2\meg1$ and let $b,\ell,k,r$ be positive integers. We say that $\mathcal{P}(d_0,d_1,d_2,b,\ell,k,r)$ holds if there exists a positive integer
  $n_0$ with the following property. For every finite alphabet $\Lambda$ of cardinality $\ell$, every integer $n$ with $n\meg n_0$, every partition $\mathbf{D}=(D_0,D_1,D_2)$ of $[d_0+d_1+d_2]$ with $|D_i|=d_i$ for all $i=0,1,2$
and every $r$-coloring $c$ of $\w^\mathbf{D}(b,n,\Lambda)$ there exists $\big((f_i)_{i=1}^d,(X_i)_{i\in D_1}\big)$ in $\w_v^\mathbf{D}(b,n,\Lambda)$ generating a $k$-dimensional $c$-good combinatorial subspace.
We denote the least such $n_0$ by $\mathrm{MTHJ}(d_0,d_1,d_2,b,\ell,k,r)$.
\end{defn}

\begin{thm}
  \label{tree_HJ}
  For every choice of non negative integers  $d_0,d_1,d_2$ with $d_0+d_1+d_2\meg1$ and every choice of positive integers $b,\ell,k,r$, we have that
  $\mathcal{P}(d_0,d_1,d_2,b,\ell,k,r)$ holds.
\end{thm}

\subsection{Proof of Theorem \ref{tree_HJ}} We have the following lemma.
\begin{lem}
  \label{branch_sensitive}
  Let $d_0,d_1$ be non negative integers with $d_0+d_1\meg1$ and let $b,\ell,k,r$ be positive integers. Then there exists a positive integer $n_0$ with the following property. For every finite alphabet $\Lambda$ of cardinality $\ell$, every integer $n$ with $n\meg n_0$, every partition $\mathbf{D}=(D_0,D_1,D_2)$ of $[d_0+d_1]$ with $D_2=\emptyset$ and $|D_i|=d_i$ for all $i=0,1$
  and every $r$-coloring $c$ of $\w^\mathbf{D}(b,n,\Lambda)$ there exists $((f_i)_{i=1}^d,(X_i)_{i\in D_1})$ in $\w_v^\mathbf{D}(b,n,\Lambda)$ generating a $k$-dimensional combinatorial subspace such that for every $((w_i)_{i\in [d]}, (x_i)_{i\in D_1}, \emptyset)$ and $((w'_i)_{i\in [d]}, (x'_i)_{i\in D_1}, \emptyset )$ in $[((f_i)_{i=1}^d,(X_i)_{i\in D_1})]_\Lambda$ with $x_i=x'_i$ and $w_i(t)=w'_i(t)$ for all $i$ in $D_1$ and $t$ in $\pred_{b^{\mik n}}(x_i)$, we have that
\[c(((w_i)_{i\in [d]}, (x_i)_{i\in D_1}, \emptyset ))=
c(((w'_i)_{i\in [d]}, (x'_i)_{i\in D_1}, \emptyset )).\]
We denote the least such $n_0$ by $h_1(d_0,d_1,b,\ell,k,r)$.
\end{lem}
\begin{proof}
  We will prove it by induction on $k$.
  For $k=1$, we work as follows. Let us fix $d_0,d_1,b,\ell,r$ as in the statement of the lemma. If $d_0=0$ then the result is trivial
  with
  \begin{equation}
    \label{eq10}
    h_1(0,d_1,b,\ell,1,r)=1.
  \end{equation}
  Assume that $d_0>0$. We will show that
  \begin{equation}
    \label{eq06}
    h_1(d_0,d_1,b,\ell,1,r)\mik d_0\mathrm{HJ}(\ell^{d_0}, r^{(b\ell)^{d_1}})-d_0+1.
  \end{equation}
  Indeed, pick any integer $n$ with $n\meg d_0\mathrm{HJ}(\ell^{d_0},r^{(b\ell)^{d_1}})-d_0+1$. Let $\Lambda,c$ and $\mathbf{D}=(D_0,D_1,D_2)$ be as in the statement of the lemma. Fix $\alpha$ in $\Lambda$.
  We set $d=d_0+d_1$ and $M=\mathrm{HJ}(\ell^{d_0},r^{(b\ell)^{d_1}})$. Pick for each $i=1,...,d$ a chain $C_i$ of $b^{<n}$ satisfying the following.
  \begin{enumerate}
    \item[(i)] For every $i$ in $D_0$ we have that $C_i$ is of cardinality $M$ and for every $i$ in $D_1$ we have that $C_i$ is a singleton.
    \item[(ii)] For every $i$ in $[d-1]$, we have that $\mathrm{h}_{b^{<n}}(\max_\sqsubseteq C_i)\mik\mathrm{h}_{b^{<n}}(\min_\sqsubseteq C_{i+1})$.
  \end{enumerate}
  Let as write $C_i=\{t_i\}$ for all $i$ in $D_1$ and $C_i=\{t^i_0\sqsubset...\sqsubset t^i_{M-1}\}$ for all $i$ in $D_0$. Moreover, for every $i$ in $D_1$ we pick $x^i_0,...,x^i_{b-1}$ in $b^n$ such that $t_i^\con(j)\sqsubseteq x^i_j$ for each $j=0,...,b-1$. Define $\Lambda'=\Lambda^{D_0}$. For each $\mathbf{a}=\big((a_q^i)_{i\in D_0}\big)_{q=0}^{M-1}$ in $\Lambda'^M$ and $\mathbf{e}=(e_i)_{i\in D_1}$ in $\Lambda^{D_1}$ we define $\mathbf{w}_{\mathbf{a},\mathbf{e}}=(w^{\mathbf{a},\mathbf{e}}_i)_{i\in[d]}$ in $\w^d(b,n,\Lambda)$ as follows.
  \begin{enumerate}
    \item[(a)] For each $i$ in $[d]$ and $t$ in $b^{<n}\setminus C_i$, we set $w^{\mathbf{a},\mathbf{e}}_i(t)=\alpha$.
    \item[(b)] For each $i$ in $D_0$ and $q$ in $\{0,...,M-1\}$, we set $w^{\mathbf{a},\mathbf{e}}_i(t_q^i)=a^i_q$.
    \item[(c)] For each $i$ in $D_1$, we set $w^{\mathbf{a},\mathbf{e}}_i(t_i)=e_i$.
  \end{enumerate}
  Finally, we define an $r^{(b\ell)^{d_1}}$-coloring $c'$ of $\Lambda'^M$ setting for each $\mathbf{a}$ in $\Lambda'^M$
  \[c'(\mathbf{a})=\big(c((w_i^{\mathbf{a},\mathbf{e}})_{i\in [d] },(x^i_{j_i})_{i\in D_1},\emptyset)\big)_{(\mathbf{e},(j_i)_{i\in D_1})\in \Lambda^{D_1}\times\{0,...,b-1\}^{D_1}}.\]
  By the choice of $M$, applying Hales--Jewett Theorem (that is, Theorem \ref{Hales_Jewett}), we have that there exists a variable word $w(v)$ generating a $c'$ monochromatic combinatorial line. Write $w(v)=(\beta_q)_{q=0}^{M-1}$. Let $Y=\mathrm{supp}_w(v)$ and $Z=\{0,...,M-1\}\setminus Y$. Moreover, let $q^*$ be the minimum of $Y$ and for each $q$ in $Z$ write $\beta_q=(a^i_q)_{i\in D_0}$. We define $(f_i)_{i\in[d]}$ in $\w^d_v(b,n,\Lambda)$ as follows.
  \begin{enumerate}
    \item[(1)] For each $i$ in $[d]$ and $t$ in $b^{<n}\setminus C_i$, we set $f_i(t)=\alpha$.
    \item[(2)] For each $i$ in $D_0$ and $q$ in $Y$, we set $f_i(t_q^i)=v_{t^i_{q^*}}$.
    \item[(3)] For each $i$ in $D_0$ and $q$ in $Z$, we set $f_i(t_q^i)=a^i_q$.
    \item[(4)] For each $i$ in $D_1$, we set $f_i(t_i)=v_{t_i}$.
  \end{enumerate}
  It is easy to check that $\big((f_i)_{i\in[d]}, (\{x^i_j:j=0,...,b-1\})_{i\in D_1}\big)$ is as desired and the
  proof of the base case ``$k=1$'' is complete.

  Assume that the statement holds for some $k$. We will prove it for $k+1$. Again we fix $d_0,d_1,b,\ell,r$ as in the statement. Set, as before, $d=d_0+d_1$. We show that
  \begin{equation}
    \label{eq07}
    h_1(d_0,d_1,b,\ell,k+1,r)\mik M+db^M Q-db^M+1,
  \end{equation}
  where $M=h_1(d_0,d_1,b,\ell,k,r^{(\ell b)^{d_1}})$ and
  $Q=\mathrm{Q}(\ell,d_1,db^M,r^{\ell^{d\frac{1-b^M}{1-b}}b^{d_1}})$
  (recall that $\mathrm{Q}(\cdot)$ is defined in the statement of Corollary \ref{Hales_Jewett_cor_v2}). Indeed, pick any integer $n$ with $n\meg M+db^MQ-db^M+1$. Let $\Lambda,c$ and $\mathbf{D}=(D_0,D_1,D_2)$ be as in the statement.
  Also let $(s_p)_{p=0}^{b^M-1}$ be the nodes in $b^M$ enumerated in $\lex$-increasing order and fix $\alpha$ in $\Lambda$. We set
  \[\mathcal{A}=[d]\times\{0,...,b^M-1\}\times\{0,...,Q-1\}\]
  and we
  define a bijection $\iota$ from $\mathcal{A}$ to $\{0,....,db^MQ-1\}$ by the rule
  \[\iota(i,p,q)=(i-1)b^MQ+p Q+q\]
  for all $(i,p,q)$ in $\mathcal{A}$.
  For every $i$ in $[d]$ and $p$ in $\{0,...,b^M-1\}$ we pick a chain $C_{i,p}$ in $b^{<n}$ of cardinality $Q$ with the following properties.
  \begin{enumerate}
    \item[(i)] For every $i$ in $[d]$ and $p$ in $\{0,...,b^M-1\}$, we have that $s_p\sqsubseteq \min_{\sqsubseteq}C_{i,p}$.
    \item[(ii)] For every $i$ in $[d]$ and $p$ in $\{0,...,b^M-2\}$, we have that
    $\mathrm{h}_{b^{<n}}(\max_\sqsubseteq C_{i,p})\mik
    \mathrm{h}_{b^{<n}}(\min_\sqsubseteq C_{i,p+1})$.
    \item[(iii)] For every $i$ in $[d-1]$, we have that $\mathrm{h}_{b^{<n}}(\max_\sqsubseteq C_{i,b^M-1})\mik
    \mathrm{h}_{b^{<n}}(\min_\sqsubseteq C_{i+1,0})$.
  \end{enumerate}
  Let us write $C_{i,p}=\{t^{i,p}_0\sqsubset...\sqsubset t^{i,p}_{Q-1}\}$ for all
  $i$ in $[d]$ and $p$ in $\{0,...,b^M-1\}$. Moreover, for every $i$ in $D_1$, $p$ in $\{0,...,b^M-1\}$ and $q$ in $\{0,...,Q-1\}$, we pick $x^{i,p}_{q,0},...,x^{i,p}_{q,b-1}$ nodes in $b^n$ such that $t^{i,p}_q\;^\con(j)\sqsubseteq x^{i,p}_{q,j}$ for all $j=0,...,b-1$.
  For every $\mathbf{a}=(a_z)_{z=0}^{db^MQ-1}$ in $\Lambda^{db^MQ}$ we define maps $w_1^\mathbf{a},...,w_d^\mathbf{a}$ from $b^{<n}\setminus b^{<M}$ to $\Lambda$ as follows.
  \begin{enumerate}
    \item[(i)] For each $i$ in $[d]$ and $t$ in $b^{<n}\setminus(b^{<M}\cup\bigcup_{p=0}^{b^M-1}(C_{i,p}))$, we set $w^\mathbf{a}_i(t)=\alpha$.
    \item[(ii)] For each $(i,p,q)$ in $\mathcal{A}$ we set $w_i^\mathbf{a}(t^{i,p}_q)=a_{\iota(i,p,q)}$.
  \end{enumerate}
  We define $I_y=\{yQ,...,(y+1)Q-1\}$ for all $y$ in $\{0,...,db^M-1\}$, while for each $i$ in $[d]$ we define $G_i=\{(i-1)b^M,...,ib^M-1\}$.
  Observe that
  \begin{equation}
    \label{eq08}
    I_{(i-1)b^M+p}=\{\iota(i,p,q):\;q=0,...,Q-1\}
  \end{equation}
  for every choice of $i$ in $[d]$ and $p$ in $\{0,...,b^M-1\}$
  and
  \begin{equation}
    \label{eq11}
    \bigcup_{y\in G_i}I_y=\{\iota(i,p,q):\;q=0,...,Q-1\;\text{and}\;p=0,...,b^M-1\}
  \end{equation}
  for every choice of $i$ in $[d]$.
  Also set
  \[\mathcal{X}=\w^d(b,M,\Lambda)\times\{(j_i)_{i\in D_1}:\;j_i\in\{0,...,b-1\}\;\text{for all}\;i\in D_1\}.\]
  We define an $r^{|\mathcal{X}|}$-coloring $c'$ of $\Lambda^{db^MQ}\times\prod_{i\in D_1}\bigcup_{y\in G_i}I_y$ setting
  for each $(\mathbf{a},\mathbf{z})=((\alpha_z)_{z=0}^{db^MQ-1},(z_i)_{i\in D_1})$ in $\Lambda^{db^MQ}\times\prod_{i\in D_1}\bigcup_{y\in G_i}I_y$
  \[c'((\mathbf{a},\mathbf{z}))=
  (c((w_i\cup w_i^\mathbf{a})_{i\in[d]},(x^{i,p_i}_{q_i,j_i})_{i\in D_1},\emptyset))_{((w_i)_{i\in[d]},(j_i)_{i\in D_1})\in\mathcal{X}},\]
  where $p_i$ and $q_i$ are the unique elements in $\{0,...,b^M-1\}$
   and $\{0,...,Q-1\}$ respectively such that $z_i=\iota(i,p_i,q_i)$, which exist due to \eqref{eq11}, for all $i$ in $D_1$.
  Since $\mathcal{X}$ is of cardinality $\ell^{d\frac{1-b^M}{1-b}}b^{d_1}$, by the definition of $Q$ and applying Corollary \ref{Hales_Jewett_cor_v2} for ``$k=\ell$'', ``$\kappa=d_1$'', ``$\lambda=b^M$'', ``$m=db^M$'' and ``$r=r^{\ell^{d\frac{1-b^M}{1-b}}b^{d_1}}$'', we obtain an $db^M$-dimensional variable word $w'(v_0,...,v_{db^M-1})$ over $\Lambda$ satisfying the following.
  \begin{enumerate}
    \item[(a)] The wildcard set of $v_y$ in $w'$ is contained in $I_{y}$ for all $y=0,...,db^M-1$.
    \item[(b)] For every $a_0,...,a_{db^M-1},a'_0,...,a'_{db^M-1}$ in $\Lambda$ and $(y_i)_{i\in D_1}$ in $\prod_{i\in D_1}G_i$ with $a_{y_i}=a'_{y_i}$ for all $i$ in $D_1$, we have that
        \[c'((w'(a_0,...,a_{db^M-1}),(\ell^{w'}_{y_i})_{i\in D_1}))=c'((w'(a'_0,...,a'_{db^M-1}),(\ell^{w'}_{y_i})_{i\in D_1})).\]
  \end{enumerate}
  Observe that due to property (a) and equation \eqref{eq08} above, we have that for every $i$ in $[d]$ and $p$ in $\{0,...,b^M-1\}$ there exists unique $q_{i,p}$ in $\{0,...,Q-1\}$ such that $\iota(i,p,q_{i,p})=\ell^{w'}_{y_{i,p}}$, where $y_{i,p}=(i-1)b^M+p$. Moreover, the map sending each $(i,p)$ to $y_{i,p}$ is a bijection between $[d]\times\{0,...,b^M-1\}$ and $\{0,...,db^M-1\}$. For every $i$ in $[d]$ and $p$ in $\{0,...,b^M-1\}$, we set $t^{i,p}=t^{i,p}_{q_{i,p}}$ and $x^{i,p}_{j}=x^{i,p}_{q_{i,p},j}$ for each $j$ in $\{0,...,b-1\}$.
  Write $w'((v_0,...,v_{db^M}-1))=(\beta_0,...,\beta_{db^MQ-1})$. We define $((f'_i)_{i\in[d]}, (X'_i)_{i\in D_1})$ in $\w^\mathbf{D}_v(b,n,\Lambda)$ as follows.
  \begin{enumerate}
    \item[(i)] For every $i$ in $[d]$ and $t$ in $b^{<M}$, we set $f_i'(t)=v_t$.
    \item[(ii)] For every $i$ in $[d]$ and $t$ in $b^{<n}\setminus(b^{<M}\cup\bigcup_{p=0}^{b^M-1}(C_{i,p}))$, we set $f_i'(t)=\alpha$.
    \item[(iii)] For every $i$ in $[d]$, $p$ in $\{0,...,b^M-1\}$ and $q$ in $\{0,...,Q-1\}$ such that $\beta_{\iota(i,p,q)}$ belongs to $\Lambda$, we set $f'_i(t^{i,p}_q)=\beta_{\iota(i,p,q)}$.
    \item[(iv)] For every $i$ in $[d]$, $p$ in $\{0,...,b^M-1\}$, $q$ in $\{0,...,Q-1\}$ and $y$ in $\{0,...,db^M-1\}$ such that $b_{\iota(i,p,q)}=v_y$, we set $f'_i(t^{i,p}_q)=v_{t^{i,p}}$.
    \item[(v)] For every $i$ in $D_1$, we set $X'_i$ to be the set of all elements $x^{i,p}_{j}$, where $p$ belongs to $\{0,...,b^M-1\}$ and $j$ belongs to $\{0,...,b-1\}$.
  \end{enumerate}
  It is easy to observe that $((f' _i)_{i\in [d]},(X'_i)_{i\in D_1})$  indeed belongs to $\w^\mathbf{D}_v(b,n,\Lambda)$.
  Let $\iota^*$ be the bijection between $b^M$ and $\{0,...,b^M-1\}$ sending each $s_p$ to $p$ and
  set \[\mathcal{Y}=\w^\mathbf{D}(b,M,\Lambda)\times\Lambda^{D_1}\times\{0,...,b-1\}^{D_1}.\]
  Observe that by the choice of $((f' _i)_{i\in [d]},(X'_i)_{i\in D_1})$, we have that for every
  $\mathbf{y}=(((w'_i)_{i\in [d]},(s'_i)_{i\in D_1},\emptyset),(a_i)_{i\in D_1},(j_i)_{i\in D_1})$ in $\mathcal{Y}$ there is some color $c_\mathbf{y}$ such that
  \[c(((w_i)_{i\in [d]},(x_i)_{i\in D_1}, \emptyset))=c_\mathbf{y}\]
  for all $((w_i)_{i\in [d]},(x_i)_{i\in D_1}, \emptyset)$ in $[((f' _i)_{i\in [d]},(X'_i)_{i\in D_1})]_\Lambda$ satisfying
  \begin{enumerate}
    \item[(i)] $w_i(t)=w'_i(t)$ for all $i$ in $[d]$ and $t$ in $b^{<M}$ and
    \item[(ii)] $w_i(t^{i,\iota^*({s'_i})})=a_i$ and $x_i=x^{i,\iota^*(s'_i)}_{j_i}$ for all $i$ in $D_1$.
  \end{enumerate}
  Set $\mathcal{Z}=\Lambda^{D_1}\times\{0,...,b-1\}^{D_1}$ and consider the $r^{|\mathcal{Z}|}$-coloring $c''$ of $\w^\mathbf{D}(b,M,\Lambda)$ defined by the rule
  \[c''(\mathbf{w}')=(c_{(\mathbf{w}',\mathbf{a},\mathbf{j})})_{(\mathbf{a},\mathbf{j})\in\mathcal{Z}},\]
  for all $\mathbf{w}'$ in $\w^\mathbf{D}(b,M,\Lambda)$. Since $\mathcal{Z}$ is of cardinality $(\ell b)^{d_1}$, by the choice of $M$, applying the inductive assumption, we obtain $((f''_i)_{i\in [d]},(S_i)_{i\in D_1})$ in $\w^\mathbf{D}_v(b,M,\Lambda)$
  generating a combinatorial subspace of dimension $k$ and having the property that for every $((w'_i)_{i\in [d]},(s'_i)_{i\in D_1},\emptyset)$ and $((\tilde w'_i)_{i\in [d]},(\tilde s'_i)_{i\in D_1},\emptyset)$ in $[((f''_i)_{i\in [d]},(S_i)_{i\in D_1})]_\Lambda$
  with $s'_i=\tilde s'_i$ and $w'_i(t)=\tilde w'_i(t)$ for all $i$ in $D_1$ and $t$ in $\pred_{b^{\mik M}}(s'_i)$,
  we have that
  \[c''(((w'_i)_{i\in [d]},(s'_i)_{i\in D_1},\emptyset))=c''(((\tilde w'_i)_{i\in [d]},(\tilde s'_i)_{i\in D_1},\emptyset)).\]
  Finally, we pick $((f_i)_{i\in [d]},(X_i)_{i\in D_1})$ in $\w^\mathbf{D}_v(b,n,\Lambda)$ satisfying the following.
  \begin{enumerate}
    \item[(i)] The combinatorial subspace generated by $((f_i)_{i\in [d]},(X_i)_{i\in D_1})$ is of dimension $k+1$.
    \item[(ii)] $[((f_i)_{i\in [d]},(X_i)_{i\in D_1})]_\Lambda\subseteq[((f'_i)_{i\in [d]},(X'_i)_{i\in D_1})]_\Lambda$.
    \item[(iii)] For every $i$ in $[d]$ and $t$ in $b^{<M}$, we have that $f_i(t)=f''_i(t).$
    \item[(iv)] For every $i$ in $D_1$ and $x$ in $X_i$ there exists $s$ in $S_i$ such that $s\sqsubseteq x$.
  \end{enumerate}
  It follows readily that $((f_i)_{i\in [d]},(X_i)_{i\in D_1})$ is as desired.
\end{proof}

\begin{prop}
  \label{prop_d_2=0}
  For every choice of non negative integers  $d_0,d_1$ with $d_0+d_1\meg1$ and every choice of positive integers $b,\ell,k,r$, we have that
  $\mathcal{P}(d_0,d_1,0,b,\ell,k,r)$ holds.
\end{prop}
\begin{proof}
  Let $d_0,d_1,b,\ell,k,r$  be as in the statement. In particular, we will show that
  \[\mathrm{MTHJ}(d_0,d_1,0,b,\ell,k,r)\mik
  h_1(d_0,d_1,b,\ell,\mathrm{MHJ}((\ell b)^{d_1},k,r),r).\]
  Indeed, let
  $M=\mathrm{MHJ}((\ell b)^{d_1},k,r)$ and $n$ an integer with $n\meg
  h_1(d_0,d_1,b,\ell,M,r)$.
  Also, let $\Lambda$ be an alphabet with $\ell$ letters. Set $d=d_0+d_1$ and let $\mathbf{D}=(D_0,D_1,D_2)$ be
  a partition of $[d]$ with $D_2=\emptyset$ and  $|D_i|=d_i$ for each $i=0,1$. Finally, let
  $c$ be an $r$-coloring of $\w^\mathbf{D}(b,n,\Lambda)$.
  By Lemma \ref{branch_sensitive}, we obtain $((f'_i)_{i\in[d]}  , (X'_i)_{i\in D_1})$ in $\w^\mathbf{D}_v(b,n,\Lambda)$ generating an $M$-dimensional combinatorial subspace such that
  for every $((w_i)_{i\in [d]}, (x_i)_{i\in D_1},
  \emptyset)$ and $((w'_i)_{i\in [d]}, (x'_i)_{i\in D_1}, \emptyset )$
  in $[((f'_i)_{i\in[d]},(X'_i)_{i\in D_1})]_\Lambda$ with $x_i=x'_i$
  and $w_i(t)=w'_i(t)$ for all $i$ in $D_1$ and $t$ in $\pred_{b^{\mik n}}(x_i)$, we have that
  \begin{equation}
    \label{eq09}
    c(((w_i)_{i\in [d]}, (x_i)_{i\in D_1}, \emptyset ))=
c(((w'_i)_{i\in [d]}, (x'_i)_{i\in D_1}, \emptyset )).
  \end{equation}
  For every $i$ in $D_1$ there exists unique $(M+1)$-complete skew subtree $T_i$ of $b^{\mik n}$ such that $X_i=T_i(M)$. In particular, we have that $\ws(f'_i)=\bigcup_{m=0}^{M-1}T_i(m)$ for each $i$ in $D_1$. For every $i$ in $D_1$ and $x$ in $X_i$ we set $y_i^x=\mathrm{I}^{-1}_{T_i}(x)$.
   We define an operator $\mathcal{Q}$ from $[((f'_i)_{i=1}^d,(X'_i)_{i\in D_1})]_\Lambda$ to $((\Lambda\times\{0,...,b-1\})^{D_1})^{M}$ by the rule
  \[\mathcal{Q}(((w_i)_{i\in [d]},(x_i)_{i\in D_1},\emptyset))=(((w_i(\mathrm{I}_{T_i}(y_i^{x_i}\upharpoonright q)),y_i^{x_i}(q)))_{i\in D_1})_{q=0}^{M-1}.\]
  The operator $\mathcal{Q}$ records the information of the elements in $[((f'_i)_{i=1}^d,(X'_i)_{i\in D_1})]_\Lambda$ that determines their color. In particular, \eqref{eq09} can equivalently be stated as follows. If two elements in $\big[\big((f'_i)_{i=1}^d,(X'_i)_{i\in D_1}\big)\big]_\Lambda$ have the same image throw $\mathcal{Q}$, then they have the same color.
  Moreover, it is easy to observe that $\mathcal{Q}$ is onto. Let us pick for every $\mathbf{a}$ in $((\Lambda\times\{0,...,b-1\})^{D_1})^{M}$ some $\mathbf{w}_\mathbf{a}$ in $\mathcal{Q}^{-1}(\mathbf{a})$. Moreover, we define an $r$-coloring $c'$ on $((\Lambda\times\{0,...,b-1\})^{D_1})^{M}$ by the rule
  \[c'(\mathbf{a})=c(\mathbf{w}_\mathbf{a})\]
  for all $\mathbf{a}$ in $((\Lambda\times\{0,...,b-1\})^{D_1})^{M}$. By the choice of $M$, applying Theorem \ref{Hales_Jewett_mult_simple},
  we obtain a $k$-dimensional variable word $w(v_0,...,v_{k-1})$ that generates a $c'$-monochromatic combinatorial subspace.
  Write $w(v_0,...,v_{k-1})=(\beta_0,...,\beta_{M-1})$ and for every $q$ in $\{0,...,M-1\}$ such that $\beta_q$ belongs to
  $(\Lambda\times\{0,...,b-1\})^{D_1}$ write $\beta_q=(a^i_q,j^i_q)_{i\in D_1}$.
  For each $i$ in $D_1$, we define two $k$-dimensional variable words
  $w^i_\Lambda(v_0,...,v_{k-1})=(\beta^{\Lambda,i}_q)_{q=0}^{M-1}$ and  $w^i_T(v_0,...,v_{k-1})=(\beta^{T,i}_q)_{q=0}^{M-1}$
  over $\Lambda$ and $\{0,...,b-1\}$ respectively of length $M$ such that for every $q$ in $\{0,...,M-1\}$, if $\beta_q$ is
  equal to some $v_p$ then we set $\beta^{\Lambda,i}_q=\beta^{T,i}_q=v_p$, while if $\beta_q$ belongs to
  $(\Lambda\times\{0,...,b-1\})^{D_1}$ then we set $\beta^{\Lambda,i}_q=a^i_q$ and $\beta^{T,i}_q=j^i_q$. For each $i$ in
  $D_1$, we set
  \[\begin{split}
    S_i=&
    \{\mathrm{I}_{T_i}(w^i_T(a_0,...,a_{k-1})\upharpoonright\ell^w_p):a_0,...,a_{k-1}\in\Lambda\;\text{and}\;p=0,...,k-1\}\\
    &\cup
  \{\mathrm{I}_{T_i}(w^i_T(a_0,...,a_{k-1})):a_0,...,a_{k-1}\in\Lambda\}
  \end{split}
  \]
  and $S_i^*=\bigcup_{p=0}^{k-1}S_i(p)$.
  It is immediate that each $S_i$ is a $(k+1)$-complete skew subtree of $T_i$. Moreover, for each $p$ in $\{0,...,k-1\}$ we have that $S_i(p)\subseteq T_i(\ell^w_p)$, while the last level $S_i(k)$ of $S_i$ is contained in the last level $T_i(M)$ of $T_i$ which is equal to $X_i$.
  For every $i$ in $D_0$, we pick a $k$-complete skew subtree $S_i^*$ of $\ws(f'_i)$ such that $S^*_i(p)\subseteq
  \ws(f'_i)(\ell^w_p)$ for all $p=0,...,k-1$. By these choices, we clearly have that $\mathbf{S}^*=(S^*_1,...,S^*_d)$ is a
  vector skew subtree of $(\ws(f'_i))_{i\in[d]}$. Fix some $\alpha$ in $\Lambda$ and for every $i$ in $D_0$, let $f_i$ be
  the unique element of $\w_v(b,n,\Lambda)$ satisfying the following.
  \begin{enumerate}
    \item[(i)] $[f_i]_\Lambda\subseteq[f'_i]_\Lambda$.
    \item[(ii)] $f_i(t)=\alpha$ for all $t$ in $\ws(f'_i)\setminus S^*_i$.
    \item[(iii)] $f_i(t)=v_t$ for all $t$ in $S^*_i$.
  \end{enumerate}
  Clearly $\ws(f_i)=S^*_i$ for all $i$ in $D_0$.

  Next, for every $i$ in $D_1$, we will define an appropriate $f_i$ in $\w_v(b,n,\Lambda)$.
  To this end, we need some preparatory definitions. Fix some $i$ in $D_1$.
  Set
  \[A_i=\bigcup_{s\in S_i(k)}\pred_{T_i}(s)\]
  and observe that
  \[A_i=\{\mathrm{I}_{T_i}(w^i_T(a_0,...,a_{k-1})\upharpoonright q):a_0,...,a_{k-1}\in\Lambda\;\text{and}\;q=0,...,M-1\}.\]
  For every $s$ in $S^*_i$, we set
  \[Y_s=\{t\in A_i:\;\mathrm{h}_{T_i}(t)\in\mathrm{supp}_{w}(v_{\mathrm{h}_{S_i}(s)})
  \;\text{and}\;s\sqsubseteq t\}.\]
  For every $i$ in $D_1$, let $f_i$ be the unique element of $\w_v(b,n,\Lambda)$ satisfying the following.
  \begin{enumerate}
    \item[(i)] $[f_i]_\Lambda\subseteq[f'_i]_\Lambda$.
    \item[(ii)] $f_i(t)=\alpha$ for all $t$ in $\ws(f'_i)\setminus A_i$.
    \item[(iii)] $f_i(t)=v_s$ for all $s$ in $S^*_i$ and $t$ in $Y_s$.
    \item[(iv)]  $f_i(t)=\beta^{\Lambda,i}_{\mathrm{h}_{T_i}(t)}$ for all $t$ in $A_i\setminus\bigcup_{s\in S^*_i}Y_s$.
  \end{enumerate}
  Let us notice that for every $t$ in $A_i\setminus\bigcup_{s\in S^*_i}Y_s$ we have that
  $\mathrm{h}_{T_i}(t)$ does not belong to $\bigcup_{p=0}^{k-1}\mathrm{supp}_w(v_p)$ and therefore $\beta^{\Lambda,i}_{\mathrm{h}_{T_i}(t)}$ is an element of $\Lambda$. Thus, it is easy to verify that $f_i$ indeed belongs to $\w_v(b,n,\Lambda)$ and $\ws(f_i)=S^*_i$ for all $i$ in $D_1$.

  Finally, setting $X_i=S_i(k)$ for all $i$ in $D_1$, we have that $((f_i)_{i\in[d]},(X_i)_{i\in D_1})$ is as desired. Indeed, it follows easily that $((f_i)_{i\in[d]},(X_i)_{i\in D_1})$ belongs to $\w^\mathbf{D}_v(b,n,\Lambda)$ with  $[((f_i)_{i\in[d]},(X_i)_{i\in D_1})]_\Lambda\subseteq [((f'_i)_{i\in[d]},(X'_i)_{i\in D_1})]_\Lambda$.
  Thus, by the definition of the color $c'$ and \eqref{eq09}, we have that
  \[c(\mathbf{w})=c'(\mathcal{Q}(\mathbf{w}))\]
  for all
  $\mathbf{w}$  in $[((f_i)_{i\in[d]},(X_i)_{i\in D_1})]_\Lambda$.
  On the other hand, it is easy to see that by the choice of $((f_i)_{i\in[d]},(X_i)_{i\in D_1})$ we have that the image
  throw $\mathcal{Q}$ of every element in $[((f_i)_{i\in[d]},(X_i)_{i\in D_1})]_\Lambda$ belongs to the combinatorial
  subspace generated by $w$, which is $c'$ monochromatic. Therefore $[((f_i)_{i\in[d]},(X_i)_{i\in D_1})]_\Lambda$ is $c$
  monochromatic as desired.
\end{proof}

\begin{lem}
  \label{tree_HJ_dim_1}
  For every choice of positive integers $d,b,\ell,r,m$, there exists a positive integer $n_0$ with the following property.
  For every $n,\Lambda,D_2, \mathbf{D}_1,...,\mathbf{D}_m$
  and $c_1,...,c_m$ such that
  \begin{enumerate}
    \item[(i)] $n$ is an integer with $n\meg n_0$,
    \item[(ii)] $\Lambda$ in an alphabet with $\ell$ elements,
    \item[(iii)] $D_2$ is a subset of $[d]$,
    \item[(iv)] $\mathbf{D}_p=(D_0^p,D_1^p,D_2^p)$ is a partition of $[d]$ with $D_1^p\cap D_2=\emptyset$ and $D_2^p\subseteq D_2$ for all $p$ in $[m]$ and
    \item[(v)] $c_p$ is an $r$-coloring of $\w^{\mathbf{D}_p}(b,n,\Lambda)$ for all $p$ in $[m]$,
  \end{enumerate}
  there exist $(f_i)_{i\in[d]}$ and $(X_i)_{i\in[d]\setminus D_2}$ such that for every $p=1,...,m$ we have that
  $((f_i)_{i\in[d]},(X_i)_{i\in D^p_1})$ belongs to $\w^{\mathbf{D}_p}_v(b,n,\Lambda)$ and generates an $1$-dimensional
  $c_p$-good combinatorial subspace.
  We denote the least such $n_0$ by $h_2(d,m,b,\ell,r)$.
\end{lem}
\begin{proof}
  Let us fix $d,b,\ell,r$ and $m$ as in the statement. We need to define some quantities in order to describe explicitly an upper bound for $h_2(d,m,b,\ell,r)$. By inverse recursion we define a sequence $(q^*_p)_{p=1}^m$ of positive integers by the rule:
    \[
        \left\{ \begin{array} {l} q_m^*=\max_{0\mik d_2'\mik d_2\mik d}\mathrm{Q}(\ell,d_2',d_2,r),\\
            q_p^*=\max_{0\mik d_2'\mik d_2\mik d}\mathrm{Q}(\ell,d_2',d_2q^*_{p+1},r),\;p=m-1,...,1,
        \end{array}  \right.
    \]
where $Q(\cdot)$ is as defined in Corollary \ref{Hales_Jewett_cor_v2}.
We set $Q_*=q^*_1$. Moreover, by inverse recursion we define a sequence $(M^*_p)_{p=1}^m$ of positive integers by the rule:
    \[
\left\{ \begin{array} {l} M_m^*=\max_{d_0+d_1+d_2=d}\mathrm{MTHJ}(d_0,d_1,0,b,\ell,1,r^{\ell^{Q_*d_2}Q_*^{d_2}}),\\
M_p^*=\max_{d_0+d_1+d_2=d}\mathrm{MTHJ}(d_0,d_1,0,b,\ell,M^*_{p+1},r^{\ell^{Q_*d_2}Q_*^{d_2}}),\;p=m-1,...,1.
\end{array}  \right.
\]
Finally, we set $M_*=M^*_1$ and $n_0=dM_*+dQ_*$. We will show that
\[h_2(d,m,b,\ell,r)\mik n_0.\]

Indeed, let $n$ be an integer with $n\meg n_0$. Also, let $\Lambda,D_2, \mathbf{D}_1,...,\mathbf{D}_m$
  and $c_1,...,c_m$ be as in the assumptions of Lemma \ref{tree_HJ_dim_1}. We set $d_2=|D_2|$ and for every $p$ in $[m]$, we set $d_0^p=|D_0^p\setminus D_2|$, $d_1^p=|D_1^p|$ and $d_2^p=|D_2^p|$.
  By inverse recursion we define a sequence $(q_p)_{p=1}^m$ of positive integers by the rule:
    \[
        \left\{ \begin{array} {l} q_m=\mathrm{Q}(\ell,d_2^m,d_2,r),\\
        q_p=\mathrm{Q}(\ell,d_2^p,d_2q_{p+1},r),\;p=m-1,...,1.
        \end{array}  \right.
    \]
We set $Q=q_1$. Clearly, $Q\mik Q_*$. Moreover, by inverse recursion we define a sequence $(M_p)_{p=1}^m$ of positive integers by the rule:
    \[
\left\{ \begin{array} {l} M_m=\mathrm{MTHJ}(d_0^m,d_1^m,0,b,\ell,1,r^{\ell^{Qd_2}Q^{d_2^m}}),\\
M_p=\mathrm{MTHJ}(d_0^p,d_1^p,0,b,\ell,M_{p+1},r^{\ell^{Qd_2}Q^{d_2^p}}),\;p=m-1,...,1.
\end{array}  \right.
\]
Set $M=M_1$ and observe that $M\mik M_*$.

For each $i$ in $[d]\setminus D_2$ we pick an $M$-complete skew subtree $S_i$ of $b^{<n}$ and for every $i$ in $D_2$ we pick a chain $S_i$ in $b^{<n}$ of length $Q$ such that
\[\max L(S_i)\mik\min L(S_{i+1})\]
for all $i$ in $[d-1]$. For each $i$ in $D_2$, let $(s^i_y)_{y=0}^{Q-1}$ be the elements of $S_i$ enumerated in $\sqsubseteq$-increasing order.
For each $i$ in $[d]\setminus D_2$, we also pick a subset $X_i^*$ of $b^n$ such that the set $S_i^*=S_i\cup X_i^*$ forms a complete skew subtree of $b^{\mik n}$.
Moreover, we set $d'=d-d_2$ and $\iota$ the unique increasing bijection from $[d']$ to $[d]\setminus D_2$. For each $p$ in $[m]$ we set $D_0'^p=\iota^{-1}(D_0^p\setminus D_2)$ and $D_1'^p=\iota^{-1}(D_1^p)$. Also let $\iota_*$ be the unique increasing bijection from $\{0,...,d_2-1\}$ to $D_2$ and $D_2'^p=\iota_*^{-1}(D_2^p)$ for each $p$ in $[m]$.

We fix some $\alpha$ in $\Lambda$ and for every $i$ in $[d]\setminus D_2$
  we define a map $Q_i:\w(b,M,\Lambda)\to\w(b,n,\Lambda)$ by the rule
  \[Q_i(\tilde{w})(t)=
\left\{ \begin{array} {l} \alpha, \;\;\;\;\;\;\;\;\;\;\;\;\;\;\;\;\;\;t\in b^{<n}\setminus S_i\\
                               \tilde{w}(\mathrm{I}_{S_i}^{-1}(t)), \;\;\;t\in S_i,\end{array}  \right.
\]
for each $\tilde{w}$ in $\w(b,M,\Lambda)$.
For every $j$ in $\{0,...,d_2-1\}$, we set
\[I_j=\{jQ,...,(j+1)Q-1\}.\]
Observe that for every $z$ in $\{0,...,d_2Q-1\}$ there exist unique $j_z$ in $\{0,...,d_2-1\}$ and $y_z$ in $\{0,...,Q-1\}$ such that $z=j_zQ+y_z$ and, in particular, $z\in I_{j_z}$.
For every $i$ in $D_2$,
 we define a map $P_i:\Lambda^{d_2Q}\to\w(b,n,\Lambda)$
by the rule
\[P_i((a_z)_{z=0}^{d_2Q-1})(t)=
\left\{ \begin{array} {l} \alpha, \;\;\;\;\;\;\;\;\;\;\;\;\;\;\;\;\;\;t\not\in S_i\\
                               a_{\iota_*^{-1}(i)Q+y}, \;\;\;\;\;t=s^i_y\;\text{for some}\;y\;\text{in}\;\{0,...,Q-1\}.\end{array}  \right.
\]
For every $p$ in $[m]$ we set $\mathcal{X}_p=\Lambda^{d_2Q}\times\prod_{j\in D_2'^p}I_j$ and we define an $r^{|\mathcal{X}_p|}$-coloring $c'_p$ on  $\w^{(D_0'^p,D_1'^p,\emptyset)}(b,M,\Lambda)$ by the rule
\[
\begin{split}
&c'_p(((\tilde{w}_j)_{j\in [d']},(x_j)_{j\in D_1'^p},\emptyset))\\
&\;\;\;\;\;\;\;\;=(c_p((Q_i(\tilde{w}_{\iota^{-1}(i)}))_{i\in [d]\setminus D_2}\cup(P_i(\mathbf{a}))_{i\in D_2},\\
&\;\;\;\;\;\;\;\;\;\;\;\;\;\;\;(\mathrm{I}_{S_i^*}(x_{\iota^{-1}(i)}))_{i\in D_1^p},(s^i_{y_{z_{\iota_*^{-1}(i)}}})_{i\in D_2^p}))_{(\mathbf{a},(z_j)_{j\in D_2'^p})\in\mathcal{X}_p}.
\end{split}\]
\begin{claim}\label{claim1}
  There exist an element $(f_j')_{j\in[d']}$ of $\w^{d'}_v(b,M,\Lambda)$ and subsets $X'_1,...,X'_{d'}$ of $b^M$ such that $((f_j')_{j\in[d']},(X_j)_{j\in D_1'^p})$ belongs to $\w^{(D_0'^p,D_1'^p,\emptyset)}_v(b,M,\Lambda)$ and generates an $1$-dimensional combinatorial subspace which is $c'_p$-monochromatic for all $p$ in $[m]$.
\end{claim}
\begin{proof}
  [Proof of claim] We define $(f^0_j)_{j\in[d']}$ by setting $f_j^0(t)=v_t$ for all $j$ in $[d']$ and $t$ in $b^{<M}$. Moreover, we set $X_1^0=...=X_{d'}^0$ to be the set of all nodes in $b^M$. We construct inductively a sequence $((f^p_j)_{j\in[d']})_{p=0}^m$ in $\w_v^{d'}(b,M,\Lambda)$ and sequences $(X_1^p)_{p=0}^m,...,(X_{d'}^p)_{p=0}^m$ of subsets of $b^M$ satisfying the following for all $p=1,...,m$.
  \begin{enumerate}
    \item[(a)] We have that $[(f^p_j)_{j\in[d']}]_\Lambda\subseteq[(f^{p-1}_j)_{j\in[d']}]_\Lambda$.
    \item[(b)] We have that $X_1^p\subseteq X_1^{p-1},...,X_{d'}^p\subseteq X_{d'}^{p-1}$.
    \item[(c)] We have that $((f^p_j)_{j\in[d']},(X_j^p)_{j\in D_1'^p})$ belongs to $\w^{(D_0'^p,D_1'^p,\emptyset)}_v(b,M,\Lambda)$ and generates an $M_p$-dimensional combinatorial subspace which is,
        in addition, $c'_p$-monochromatic.
  \end{enumerate}
  Assume that for some $p$ in $[m]$, we have chosen the sequences $((f^{p'}_j)_{j\in[d']})_{p'=0}^{p-1}$ and
  $(X_1^{p'})_{p'=0}^{p-1},...,(X_{d'}^{p'})_{p'=0}^{p-1}$ satisfying (a), (b) and (c) above. Recall that $c'_p$ is an
  $r^{|\mathcal{X}_p|}$-coloring and observe that $\mathcal{X}_p$ is of cardinality $\ell^{d_2Q}Q^{d_2^p}$. Also, observe
  that the spaces $[((f_j^{p-1})_{j\in[d']},(X^{p-1}_j)_{j\in D_1'^p})]_\Lambda$ and
  $\w^{(D_0'^p,D_1'^p,\emptyset)}_v(b,M_{p-1},\Lambda)$ are isomorphic and therefore, by the definition of the number
  $M_{p-1}$ and Proposition \ref{prop_d_2=0}, there exists $((f^p_j)_{j\in[d']},(X_j^p)_{j\in D_1'^p})$ that belongs to
  $\w^{(D_0'^p,D_1'^p,\emptyset)}_v(b,M,\Lambda)$, generates an $M_p$-dimensional combinatorial subspace which is
  $c'_p$-monochromatic, satisfies (a) and partially (b) (more precisely, we have that $X^p_j\subseteq X^{p-1}_j$, for all
  $j$ in $D_1'^p$). Picking for each $j$ in $[d']\setminus D_1'^p$ a subset $X_j^p$ of $X_j^{p-1}$ such that $\ws(f_j^p)\cup
  X_j^p$ forms a skew subtree of $b^{\mik M}$, the inductive step of the selection is complete. Finally, we set $f_j'=f_j^m$ and $X_j'=X_j^m$
  for all $j$ in $[d']$. The proof of the claim is complete.
\end{proof}
By Claim \ref{claim1} we can naturally induce colorings $c_1'',...,c_m''$ on $\mathcal{X}_1,...,\mathcal{X}_m$ respectively. More precisely, for each $p$ in $[m]$ let $(\gamma^p_{(\mathbf{a},\mathbf{z})})_{(\mathbf{a},\mathbf{z})\in\mathcal{X}_p}$ be the color that $c'_p$ attaches to each element of $[((f_j')_{j\in[d']},(X_j)_{j\in D_1'^p})]_\Lambda$. Then we define
$c''_p((\mathbf{a},\mathbf{z}))=\gamma^p_{(\mathbf{a},\mathbf{z})}$
for each $p$ in $[m]$ and $(\mathbf{a},\mathbf{z})$ in $\mathcal{X}_p$.
Clearly $c''_p$ is an $r$-coloring on $\mathcal{X}_p$.
\begin{claim}
  \label{claim2}
  There exists a $d_2$-dimensional word $w(v_0,...,v_{d_2-1})$ over $\Lambda$ of length $d_2Q$ satisfying the following.
  \begin{enumerate}
    \item[(1)] The wildcard set of $v_j$ in $w$ is contained in $I_j$ for all $j=0,...,d_2-1$.
    \item[(2)] For each $p$ in $[m]$ and for every choice of $a_0,...,a_{d_2-1}$ and $a'_0,...,a'_{d_2-1}$ in $\Lambda$ such that $a_j=a'_j$ for all $j$ in $D_2'^p$, we have that
        \[c''_p((a_0,...,a_{d_2-1},(\ell^w_j)_{j\in D_2'^p}))
        =c''_p((a'_0,...,a'_{d_2-1},(\ell^w_j)_{j\in D_2'^p})).\]
  \end{enumerate}
\end{claim}
\begin{proof}
  [Proof of Claim]
  The claim follows easily by an iterated use of Corollary \ref{Hales_Jewett_cor_v2}.
\end{proof}
We define for each $i$ in $[d]\setminus D_2$ an element $f_i$ of $\w_v(b,n,\Lambda)$ as follows.
\begin{enumerate}
  \item[(i$'$)] For each $t$ in $b^{<n}\setminus S_i$, we set $f_i(t)=\alpha$.
  \item[(ii$'$)] For each $t$ is $S_i$ such that $f'_{\iota^{-1}(i)}(\mathrm{I}_{S_i}^{-1}(t))$ belongs to $\Lambda$, we set
      $f_i(t)=f'_{\iota^{-1}(i)}(\mathrm{I}_{S_i}^{-1}(t))$.
  \item[(iii$'$)] For each $t$ is $S_i$ and $s$ in $b^{<M}$ such that $f'_{\iota^{-1}(i)}(\mathrm{I}_{S_i}^{-1}(t))=v_s$, we set $f_i(t)=v_{\mathrm{I}_{S_i}(s)}$.
\end{enumerate}
We write $w(v_0,...,v_{d_2-1})=(\beta_z)_{z=0}^{d_2Q-1}$. Moreover,
for every $i$ in $D_2$, we set $y_i^*=y_{\ell^w_{j}}$, where $j=\iota_*^{-1}(i)$.
For every $i$ in $D_2$ we define
an element $f_i$ of $\w_v(b,n,\Lambda)$ as follows.
\begin{enumerate}
  \item[(i$'$)] For each $t$ in $b^{<n}\setminus S_i$, we set $f_i(t)=\alpha$.
  \item[(ii$'$)] For each $y$ in $\{0,...,Q-1\}$ such that
      $\beta_{\iota_*^{-1}(i)Q+y}$ belongs to $\Lambda$, we set $f_i(s^i_y)=\beta_{\iota_*^{-1}(i)Q+y}$.
  \item[(iii$'$)] For each $y$ in $\{0,...,Q-1\}$ such that
      $\beta_{\iota_*^{-1}(i)Q+y}=v_{\iota_*^{-1}(i)}$, we set $f_i(s^i_{y})=v_{s^i_{y_i^*}}$.
\end{enumerate}
Finally, for each $i$ in $[d]\setminus D_2$, we set $X_i=I^{-1}_{S_i}(X'_{i^{-1}(i)})$. It follows readily by the definitions of the colorings $c'_1,...,c'_m, c''_1,...,c''_m$ and Claims
\ref{claim1} and \ref{claim2} that $(f_i)_{i\in[d]}$ and $(X_i)_{i\in[d]\setminus D_2}$ satisfy the conclusion of Lemma \ref{tree_HJ_dim_1}.
\end{proof}

Lemma \ref{tree_HJ_dim_1} has the following immediate consequence.

\begin{cor}
  \label{cor_k=1}
  For every choice of non negative integers  $d_0,d_1,d_2$ with $d_0+d_1+d_2\meg1$ and every choice of positive integers $b,\ell,r$, we have that
  $\mathcal{P}(d_0,d_1,d_2,b,\ell,1,r)$ holds.
\end{cor}
\begin{proof}
  [Proof of Theorem \ref{tree_HJ}]
  The proof follows by a double induction scheme on $d_2$ and $k$.
  In particular, by Proposition \ref{prop_d_2=0},
  we have that
  $\mathcal{P}(d_0,d_1,0,b,\ell,k,r)$ holds,
  for every choice of non negative integers $d_0,d_1$ with $d_0+d_1\meg1$ and every choice of positive integers $b,\ell,k,r$.
  Moreover, by Corollary \ref{cor_k=1}, we have that $\mathcal{P}(d_0,d_1,d_2,b,\ell,1,r)$ holds,
  for every choice of non negative integers  $d_0,d_1,d_2$ with $d_0+d_1+d_2\meg1$ and every choice of positive integers $b,\ell,r$.
  We complete the proof by showing that $\mathcal{P}(d_0,d_1,d_2,b,\ell,k+1,r)$ holds, where $d_0,d_1$ are non negative integers and $d_2,b,\ell,k,r$ are positive integers, assuming that
  \begin{enumerate}
    \item[(i)] $\mathcal{P}(d_0',d_1',d_2',b,\ell,k',r')$ holds for every choice of non negative integers $d_0',d_1',d_2'$ with $d_0'+d_1'+d_2'\meg1$ and $d_2'<d_2$ and every choice of positive integers $k',r'$ and
    \item[(ii)] $\mathcal{P}(d_0',d_1',d_2,b,\ell,k',r')$  holds for every choice of non negative integers $d_0',d_1'$ and every choice of positive integers $k',r'$ with $k'\mik k$.
  \end{enumerate}

  Indeed, let us fix non negative integers $d_0,d_1$ and positive integers $d_2,b,\ell,k,r$. In order to produce an explicit upper bound for $\mathrm{MTHJ}(d_0,d_1,d_2,b,\ell,k+1,r)$ we define inductively a sequence of numbers $\big((M^{d_2'}_q)_{q=1}^{{d_2\choose d_2'}}\big)_{d_2'=0}^{d_2}$ by the rule
  \[
  \left\{ \begin{array} {l} M_1^0=\mathrm{MTHJ}(d_0,d_1,d_2,b,\ell,k,r),\\
M_1^{d_2'}=\mathrm{MTHJ}(d_0,d_1+d_2',d_2-d_2',b,\ell,M_{{d_2\choose d_2'-1}}^{d_2'-1},r^{\ell^{d_2}}),\;d_2'=1,...,d_2\\
M_{q+1}^{d_2'}=\mathrm{MTHJ}(d_0,d_1+d_2',d_2-d_2',b,\ell,M_q^{d_2'},r^{\ell^{d_2}}),\;q=1,...,{d_2\choose d_2'}-1.
\end{array}  \right.
\]
Set $M=M^{d_2}_1$. Moreover, we set
\[M'=h_2(b^M(d_0+d_1+d_2),b^{d_1M}(b^M+1)^{d_2},b,\ell,r^{\ell^{(d_0+d_1+d_2)\frac{b^M-1}{b-1}}\big(\frac{b^M-1}{b-1}\big)^{d_2}}).\]
  We will show that
  \[\mathrm{MTHJ}(d_0,d_1,d_2,b,\ell,k+1,r)\mik M+M'.\]

  Indeed, let $n$ be an integer with $n\meg M+M'$.
  Also let $\Lambda$ be an alphabet of cardinality $\ell$, let $\mathbf{D}=(D_0,D_1,D_2)$ be a partition of $[d_0+d_1+d_2]$ with $|D_i|=d_i$ for all $i=0,1,2$
and let $c$ be an $r$-coloring $c$ of $\w^\mathbf{D}(b,n,\Lambda)$. We need to find $(f_i)_{i\in[d]}$ in $\w_v^d(b,n,\Lambda)$ and a collection $(X_i)_{i\in D_1}$ such that $((f_i)_{i\in[d]},(X_i)_{i\in D_1})$ belongs to $\w_v^\mathbf{D}(b,n,\Lambda)$ and generates a $(k+1)$-combinatorial subspace which is $c$-good.
Roughly speaking, we will first construct the last level of the desired combinatorial subspace, using Lemma \ref{tree_HJ_dim_1}, by refining the last $n-M$ levels of our initial space and afterwards, using the inductive assumptions, we will refine the first $M$ levels of the initial space to obtain the first $k$ levels of the desired combinatorial subspace.

We set $d=d_0+d_1+d_2$ and $d'=db^M$. Also, for every $i$ in $[d]$ we set
\[J_i=\{(i-1)b^M+1,...,ib^M\}.\]
Clearly $(J_i)_{i\in[d]}$ forms a partition of $[d']$ into intervals of length $b^M$. Thus, for every $z$ in $[d']$ there exist unique $i_z$ in $[d]$ and $y_z$ in $[b^M]$ such that $z=(i_z-1)b^M+y_z$ and therefore $z\in J_{i_z}$. We also set
\[\mathcal{A}=\{D_1'\subseteq\bigcup_{i\in D_1}J_i:\;|D_1'\cap J_i|=1\;\text{for all}\;i\in D_1\}\]
and
\[\mathcal{B}=\{D_2'\subseteq\bigcup_{i\in D_2}J_i:\;|D_1'\cap J_i|\mik1\;\text{for all}\;i\in D_2\}.\]
Moreover, for every $D_2'$ in $\mathcal{B}$, we set
\[\mathrm{Pr}(D_2')=\{i\in D_2:\;D_2'\cap J_i\neq\emptyset\}\;\text{and}\;\overline{\mathrm{Pr}}(D_2')=D_2\setminus\mathrm{Pr}(D_2').\]
Let $(s_y)_{y=1}^{b^M}$ be an enumeration of the nodes in $b^M$ in $\lex$-increasing order.
For each $D'_1$ in $\mathcal{A}$ and $i$ in $D_1$, we denote by $z_{D'_1,i}$ the unique element $z$ in $D_1'$ such that $i=i_z$ and we set $\iota_1(D_1',i)=s_{y_{z_{D'_1,i}}}$.
Similarly, for each $D'_2$ in $\mathcal{B}$ and $i$ in $\mathrm{Pr}(D_2')$, we denote by $z_{D'_2,i}$ the unique element $z$ in $D_2'$ such that $i=i_z$ and we set $\iota_2(D_2',i)=s_{y_{z_{D'_2,i}}}$.
We also define a map
\[\iota_1:\mathcal{A}\to (b^M)^{D_1} \]
setting
\[\iota_1(D_1')=(\iota_1(D_1',i))_{i\in D_1}\]
for each $D_1'$ in $\mathcal{A}$. Similarly, we may define a map \[\iota_2:\mathcal{B}\to \bigcup_{\tilde{D}_2\subseteq D_2}(b^M)^{\tilde{D}_2}\]
setting
\[\iota_2(D_2')=(\iota_2(D_2',i))_{i\in \mathrm{Pr}(D_2')}\]
for each $D_2'$ in $\mathcal{B}$.

For every $D_2'$ in $\mathcal{B}$, we define a partition \[\mathbf{D}(D_2')=([d]\setminus\overline{\mathrm{Pr}}(D_2'),\emptyset,\overline{\mathrm{Pr}}(D_2'))\] of $[d]$.
Moreover, for every $(D_1',D_2')$ in $\mathcal{A}\times\mathcal{B}$, we define a partition
\[\mathbf{D}'(D_1',D_2')=([d']\setminus(D_1'\cup D_2'),D_1',D_2')\] of $[d']$,
and a map
\[Q_{D_1',D_2'}:\w^{\mathbf{D}(D_2')}(b,M,\Lambda)\times\w^{\mathbf{D}'(D_1',D_2')}(b,n-M,\Lambda)
\to \w^{\mathbf{D}}(b,n,\Lambda)\]
as follows. Pick  an element $\tilde{\mathbf{w}}=((\tilde{w}_{i})_{i\in[d]},\emptyset,(\tilde{t}_{i})_{i\in\overline{\mathrm{Pr}}(D_2')})$
of $\w^{\mathbf{D}(D_2')}(b,M,\Lambda)$ and an element $\mathbf{w}'=((w'_z)_{z\in[d']},(x'_z)_{z\in D_1'},(t'_z)_{z\in D_2'}))$ of $\w^{\mathbf{D}'(D_1',D_2')}(b,n-M,\Lambda)$.
For each $i$ in $[d]$, we define $w_i^{\tilde{\mathbf{w}},\mathbf{w}'}$ in $\w(b,n,\Lambda)$ as follows.
\begin{enumerate}
  \item[(a)] For every $s$ in $b^{<M}$, we set $w_i^{\tilde{\mathbf{w}},\mathbf{w}'}(s)=\tilde{w}_i(s)$.
  \item[(b)] For every $s$ in $b^{<n}\setminus b^{<M}$, there exist unique $y$ in $[b^M]$ and $t'$ in $b^{<n-M}$ such that
        $s_y\sqsubseteq s$ and $s=s_y^\con t'$. We set $w_i^{\tilde{\mathbf{w}},\mathbf{w}'}(s)=w'_{(i-1)b^M+y}(t')$.
\end{enumerate}
For every $i$ in $D_1$, we set $x_i^{\tilde{\mathbf{w}},\mathbf{w}'}=\iota_1(D_1',i)^\con x'_{z_{D_1',i}}$. For every $i$ in $\mathrm{Pr}(D_2')$, we set
$t_i^{\tilde{\mathbf{w}},\mathbf{w}'}=\iota_2(D_2',i)^\con t'_{z_{D_2',i}}$. For every $i$ in $\overline{\mathrm{Pr}}(D_2')$, we set $t_i^{\tilde{\mathbf{w}},\mathbf{w}'}=\tilde{t}_i$. Finally, we define
\[Q_{D_1',D_2'}(\tilde{\mathbf{w}},\mathbf{w}')=((w_i^{\tilde{\mathbf{w}},\mathbf{w}'})_{i\in[d]},(x_i^{\tilde{\mathbf{w}},\mathbf{w}'})_{i\in D_1},(t_i^{\tilde{\mathbf{w}},\mathbf{w}'})_{i\in D_2}).\]
It is easy to see that the map $Q_{D_1',D_2'}$ is injective and the images of $Q_{D_1',D_2'}$ as $(D_1',D_2')$ varies in
$\mathcal{A}\times\mathcal{B}$ form a partition of $\w^\mathbf{D}(b,n,\Lambda)$. Roughly speaking, the inverse of each map of the form $Q_{D_1',D_2'}$
splits the elements of the space $\w^{\mathbf{D}}(b,n,\Lambda)$ at the level $M$ and decomposes their part on the last $n-M$
levels into $d'$ pieces, one piece for each $i$ in $[d]$ and $s$ in $b^M$.

For every $(D_1',D_2')$ in $\mathcal{A}\times\mathcal{B}$, we define
an $r^{|\w^{\mathbf{D}(D_2')}(b,M,\Lambda)|}$-coloring $c_{D_1',D_2'}$ on $\w^{\mathbf{D}'(D_1',D_2')}(b,n-M,\Lambda)$ by the rule
\[c_{D_1',D_2'}(\mathbf{w}')=(c(Q_{D_1',D_2'}(\tilde{\mathbf{w}},\mathbf{w}')))_{\tilde{\mathbf{w}}\in\w^{\mathbf{D}(D_2')}(b,M,\Lambda)}.\]
Since the set $\w^{\mathbf{D}(D_2')}(b,M,\Lambda)$ is of cardinality $\ell^{d\frac{b^M-1}{b-1}}\Big(\frac{b^M-1}{b-1}\Big)^{|D_2'|}$,
the set $\mathcal{A}\times\mathcal{B}$ is of cardinality $b^{d_1M}(b^M+1)^{d_2}A$
 and
$M'\mik n-M$, by the definition of $M'$, applying Lemma \ref{tree_HJ_dim_1}, we obtain $(f'_z)_{z\in[d']}$ and $(X'_z)_{z\in[d']\setminus D_2^*}$, where $D_2^*=\bigcup_{i\in D_2}J_i$, such that
$((f'_z)_{z\in[d]},(X'_z)_{z\in D'_1})$ belongs to $\w^{\mathbf{D}'(D_1',D_2')}_v(b,n-M,\Lambda)$ and generates an $1$-dimensional $c_{D_1',D_2'}$-good combinatorial subspace for all $(D_1',D_2')$ in $\mathcal{A}\times\mathcal{B}$. Let us denote by $s'_z$ the unique element of $\ws(f'_z)$ for all $z$ in $[d']$.

By the choice of $(f'_z)_{z\in[d']}$ and $(X'_z)_{z\in[d']\setminus D_2^*}$, for every $(D'_1,D'_2)$ in $\mathcal{A}\times\mathcal{B}$ and every $\mathbf{a}=(a_z)_{z\in D_2'}$ in $\Lambda^{D'_2}$ there exists $\bar{\gamma}^{D'_1,D'_2}_\mathbf{a}=(\gamma^{D'_1,D'_2}_{\mathbf{a},\tilde{\mathbf{w}}})_{\tilde{\mathbf{w}}\in\w^{\mathbf{D}(D_2')}(b,M,\Lambda)}$ such that
\[c_{D'_1,D'_2}(\mathbf{w}')=\bar{\gamma}^{D'_1,D'_2}_\mathbf{a}\]
for all $\mathbf{w}'=((w'_z)_{z\in[d']},(x'_z)_{z\in D_1'},(t'_z)_{z\in D_2'}))$ in  $[(f'_z)_{z\in[d']},(X'_z)_{z\in D'_1}]_\Lambda$ satisfying $w'_z(s'_z)=a_z$
for all $z$ in $D'_2$ (also notice that since the combinatorial subspace $[(f'_z)_{z\in[d']},(X'_z)_{z\in D'_1}]_\Lambda$ is 1-dimensional, we have that $t'_z=s'_z$ for all $z$ in $D_2'$).

For every subset $\tilde{D}_2$ of $D_2$ we define on $\w^{(D_0,D_1\cup\tilde{D}_2,D_2\setminus \tilde{D}_2)}(b,M,\Lambda)$
an $r^{\ell^{|\tilde{D}_2|}}$-coloring $c_{\tilde{D}_2}$  by the rule
\[c_{\tilde{D}_2}(((\tilde{w}_i)_{i\in [d]},(\tilde{x}_i)_{i\in D_1\cup \tilde{D}_2},(\tilde{t}_i)_{i\in D_2\setminus\tilde{D}_2}))=(\gamma^{D_1',D_2'}_{\mathbf{a},\tilde{\mathbf{w}}})_{\mathbf{a}\in\Lambda^{D_2'}},\]
where $\tilde{\mathbf{w}}=((\tilde{w}_i)_{i\in [d]},\emptyset,(\tilde{t}_i)_{i\in D_2\setminus\tilde{D}_2})$, $D_1'=\iota_1^{-1}((\tilde{x}_i)_{i\in D_1})$ and $D_2'=\iota_2^{-1}((\tilde{x}_i)_{i\in \tilde{D}_2})$.

We set
\[\Gamma=\bigcup_{d_2'=0}^{d_2}\{d'_2\}\times\Bigg[{d_2\choose d'_2}\Bigg].\] Observe that $\Gamma$ is of cardinality
$2^{d_2}$. Let $\iota:\Gamma\to[2^{d_2}]$ be a bijection such that for every $(d_2',q)$ and $(\tilde{d}_2',\tilde{q})$ in
$\Gamma$ with $d'_2<\tilde{d}'_2$, or $d'_2=\tilde{d}'_2$ and $q<\tilde{q}$, we have that
$\iota((d_2',q))<\iota((\tilde{d}_2',\tilde{q}))$. Moreover, for every $y$ in $[2^{d_2}]$, we denote by $\mathrm{d_2'}(y)$ and $\mathrm{q}(y)$
the numbers satisfying $\iota(\mathrm{d_2'}(y),\mathrm{q}(y))=y$.
Also let $(\tilde{D}_2^y)_{y=1}^{2^{d_2}}$ be an enumeration of all subsets of $D_2$ such that for every $y$ in $[2^{d_2}]$,
the set $\tilde{D}_2^y$ is of cardinality $\mathrm{d_2'}(y)$. Let us point out that only $\tilde{D}_2^1$ is empty.

For each $i$ in $[d]$ we define $\tilde{f}_i^{2^{d_2}+1}$ in $\w_v(b,M,\Lambda)$ setting $\tilde{f}_i^{2^{d_2}+1}(t)=v_t$ for all $t$ in $b^{<M}$. Moreover, for every $i$ in $D_1\cup D_2$, we set $\tilde{X}_i^{2^{d_2}+1}$ to be the set of all nodes in $b^{M}$.
Clearly $(\tilde{f}_i^{2^{d_2}+1})_{i\in[d]}$ is of height $M=M^{d_2}_1=M^{\mathrm{d'_2}(2^{d_2})}_{\mathrm{q}(2^{d_2})}$ and for every $i$ in $[d]$ we
have that $\ws(\tilde{f}_i^{2^{d_2}+1})\cup \tilde{X}_i^{2^{d_2}+1}$ is a complete skew subtree of $b^{\mik M}$.
We construct sequences $((\tilde{f}_i^p)_{i\in[d]})_{p\in[2^{d_2}]}$ and $((\tilde{X}_i^p)_{i\in D_1\cup
D_2})_{p\in[2^{d_2}]}$ satisfying for every $p$ in $[2^{d_2}]$ the following.
\begin{enumerate}
  \item[(1)] We have that $(\tilde{f}_i^p)_{i\in[d]}$ belongs to $\w_v^d(b,M,\Lambda)$.
  \item[(2)] We have that   $[(\tilde{f}_i^p)_{i\in[d]}]_{\Lambda}\subseteq[(\tilde{f}_i^{p+1})_{i\in[d]}]_\Lambda$.
  \item[(3)] The set $\tilde{X}_i^p$ is a subset of $\tilde{X}_i^{p+1}$ for all $i$ in $D_1\cup D_2$.
  \item[(4)] The set $\mathrm{ws}(\tilde{f}_i^p)\cup \tilde{X}_i^p$ is a complete skew subtree of $b^{\mik M}$ for all $i$ in $D_1\cup D_2$.
  \item[(5)] If $p>1$, then  $(\tilde{f}_i^p)_{i\in[d]}$ is of height $M^{\mathrm{d'_2}(p-1)}_{\mathrm{q}(p-1)}$, while $(\tilde{f}_i^1)_{i\in[d]}$ is of height $k$.
  \item[(6)] The combinatorial subspace generated by $((\tilde{f}_i^p)_{i\in[d]},(\tilde{X}_i^p)_{i\in D_1\cup\tilde{D}_2^p})$ is $c_{\tilde{D}_2^p}$-good.
\end{enumerate}
The construction follows easily, by the definition of the numbers $((M^{d_2'}_q)_{q=1}^{{d_2\choose d_2'}})_{d_2'=0}^{d_2}$, applying inverse induction for $p=2^{d_2},2^{d_2}-1,...,1$ and using the inductive assumptions. Setting $\tilde{f}_i=\tilde{f}_i^1$ for all $i$ in $[d]$ and $\tilde{X}_i=\tilde{X}_i^1$ for all $i$ in $D_1\cup D_2$, we have that $((\tilde{f}_i)_{i\in[d]},(\tilde{X}_i)_{i\in D_1\cup\tilde{D}_2})$ generates a $k$-dimensional combinatorial subspace which is $c_{\tilde{D}_2}$-good
for each subset $\tilde{D}_2$ of $D_2$. Moreover, for each $i$ in $D_0$, we pick arbitrarily a subset $\tilde{X}_i$ of $b^{M}$ such that $\ws(\tilde{f}_i)\cup\tilde{X}_i$ is a complete skew subtree of $b^{\mik M}$.

Finally, we are ready to define the objects that will complete the proof of the theorem.
For every $i$ in $[d]$ we define the following. Let $y^i_1<...<y^i_{b^k}$ in $[b^M]$ such that $\tilde{X}_i=\{s_{y^i_j}:j\in[b^k]\}$. Also set
\[
\begin{split}
  &Y_i=\{y^i_j:j\in[b^k]\},\;\bar{Y}_i=[b^k]\setminus Y_i, \\ &Z_i=\{(i-1)b^M+y:y\in Y_i\}\;\text{and}\;\bar{Z}_i=J_i\setminus Z_i.
\end{split}
\]
For every $z$ in $\bar{Z}_i$, we pick an element $w'_z$ of $[f'_z]_\Lambda$. We define $f_i$ as follows.
\begin{enumerate}
  \item[(i)] For every $t$ in $b^{<M}$, we set $f_i(t)=\tilde{f}_i(t)$.
  \item[(ii)] For every $t$ in $b^{n-M}$ and $y$ in $\bar{Y}_i$, we set $f_i(s_y^\con t)=w'_{(i-1)b^M+y}(t)$.
  \item[(iii)] For every $t$ in $b^{n-M}$ and $y$ in $Y_i$ such that $f'_{(i-1)b^M+y}(t)$ belongs to $\Lambda$, we set $f_i(s_y^\con t)=f'_{(i-1)b^M+y}(t)$.
  \item[(iv)] For every $t$ in $b^{n-M}$ and $y$ in $Y_i$ such that $f'_{(i-1)b^M+y}(t)=v_{s'_{(i-1)b^M+y}}$, we set $f_i(s_y^\con t)=v_{s_y^\con s'_{(i-1)b^M+y}}$.
\end{enumerate}
Finally, for every $i$ in $D_1$, we set
\[X_i=\{s_y^\con x':\;y\in Y_i\;\text{and}\;x'\in X'_{(i-1)b^M+y}\}.\]
It follows by the definitions of the colorings $(c_{D'_1,D'_2})_{(D'_1,D'_2)\in\mathcal{A}\times\mathcal{B}}$ and $(c_{\tilde{D}_2})_{\tilde{D}_2\subseteq D_2}$, as well as, by the choice of $(f'_z)_{z\in[d']}$, $(X'_z)_{z\in[d']\setminus D_2^*}$,
$(\tilde{f}_i)_{i\in[d]}$ and $(\tilde{X}_i)_{i\in[d]}$ that $((f_i)_{i\in[d]},(X_i)_{i\in D_1})$ belongs to $\w_v^{\mathbf{D}}(b,n,\Lambda)$ generating a $(k+1)$-dimensional $c$-good subspace as desired. The proof of Theorem \ref{tree_HJ} is complete.
\end{proof}

\subsection{Consequences}
We will actually use an immediate consequence of Theorem \ref{tree_HJ} concerning smooth colorings of some special type
of combinatorial subspaces. Below we state the relevant definitions.

\begin{defn}
  Let $d,b,n$ be positive integers, let $\mathbf{D}=(D_0,D_1,D_2)$ be a partition of $[d]$ and let $\Lambda$ be a finite
  alphabet. Set $d_2 = |D_2|$ and write $D_2= \{r_1<r_2<...<r_{d_2}\}$. We denote by $\w^\mathbf{D}_*(b,n,\Lambda)$ the set of all $((w_i)_{i\in[d]}, (x_i)_{i\in D_1}, (t_i)_{i\in D_2})$ in $\w^\mathbf{D}(b,n,\Lambda)$ such that  $(t_{r_i})_{i=1}^{d_2}$ is an element of
  $\mathcal{CT}_1(\mathbf{T}^{d_2}_{b,n})$.
  Moreover, for every $\big((f_i)_{i=1}^d,(X_i)_{i\in D_1}\big)$ in $\w_v^\mathbf{D}(b,n,\Lambda)$, we define
  \[\big[\big((f_i)_{i=1}^d,(X_i)_{i\in D_1}\big)\big]_\Lambda^* =
  \big[\big((f_i)_{i=1}^d,(X_i)_{i\in D_1}\big)\big]_\Lambda \cap
  \w^\mathbf{D}_*(b,n,\Lambda).\]
\end{defn}

\begin{defn}
  Let $d,b,n$ be positive integers, let $\mathbf{D}=(D_0,D_1,D_2)$ be a partition of $[d]$ and let $\Lambda$ be a finite alphabet.
  We say that a finite coloring $c$ of $\w^\mathbf{D}(b,n,\Lambda)$ is smooth if
  there exist disjoint subsets $\Gamma_1$ and $\Gamma_2$ (possibly empty) of $D_2$
  with $\Gamma_1\cup\Gamma_2 = D_2$ satisfying the following property. For every two elements $((w_i)_{i\in[d]}, (x_i)_{i\in D_1}, (t_i)_{i\in D_2})$ and $((w'_i)_{i\in[d]}, (x'_i)_{i\in D_1}, (t'_i)_{i\in D_2} )$ of $\w^\mathbf{D}(b,n,\Lambda)$ such that
  \begin{enumerate}
    \item[(a)] $w_i(t)=w'_i(t)$ for all $i$ in $D_0\cup D_1 \cup \Gamma_1$ and $t$ in $b^{<n}$,
    \item[(b)] $x_i=x'_i$ for all $i$ in $D_1$ and
    \item[(c)] $t_i=t'_i$ and $w_i(t)=w'_i(t)$ for all $i$ in $\Gamma_2$ and $t$ in $b^{<n}\setminus \mathrm{Succ}_{b^{<n}}(t_i)$,
  \end{enumerate}
  we have that
\[c(((w_i)_{i\in[d]}, (x_i)_{i\in D_1}, (t_i)_{i\in D_2} ))=
c(((w'_i)_{i\in[d]}, (x'_i)_{i\in D_1}, (t'_i)_{i\in D_2} )).\]
\end{defn}

\begin{cor}
  \label{tree_HJ_cor}
  Let $d_0,d_1,d_2$ be non negative integers with $d_0+d_1+d_2\meg1$ and
  let $b,\ell,k,r$ be positive integers. For every alphabet $\Lambda$ of cardinality $\ell$, every integer $n$ with
  \[n\meg \mathrm{MTHJ}(d_0,d_1,d_2,b,\ell,\mathrm{CT}(1,k,b,d_2,r),r),\]
   every partition $\mathbf{D}=(D_0,D_1,D_2)$ of $[d_0+d_1+d_2]$ with $|D_i|=d_i$ for all $i=0,1,2$
and every smooth $r$-coloring $c$ of $\w^\mathbf{D}(b,n,\Lambda)$,
there exists $\big((f_i)_{i=1}^d,(X_i)_{i\in D_1}\big)$ in $\w_v^\mathbf{D}(b,n,\Lambda)$ generating a $k$-dimensional combinatorial subspace such that the set $[\big((f_i)_{i=1}^d,(X_i)_{i\in D_1}\big)]_\Lambda^*$ is monochromatic.
\end{cor}

\section{Main result}
This section is devoted to the proof of Theorem \ref{point_subsets} below which is equivalent to Theorem \ref{tree_GR} stated in the introduction.

\subsection{Semi-complete skew subtrees and related variable word spaces}

Let $b,n$ be positive integers, let $\Lambda$ be a finite alphabet and let $T$ be a skew subtree.
A $T$-variable word $f$ of $\w(b,n,\Lambda)$ is a map from $b^{<n}$ into $\Lambda\cup\{v_t:t\in T\}$ such that for every $t$ in $T$ the set $f^{-1}(v_t)$ is non empty and has $t$ as a $\sqsubseteq$-minimum.
We denote by $\w_{v,T}^\times(b,n,\Lambda)$ the set of all $T$-variable words of $\w(b,n,\Lambda)$.
Notice that the definition is identical with the case when $T$ is a complete skew subtree.
For a positive integer $l$, we denote by $\w_{v,l}^\times(b,n,\Lambda)$
the union of $\w_{v,T}(b,n,\Lambda)$ over all possible skew subtrees $T$ of $b^{<n}$ of cardinality $l$.
We also set $\w_{v,0}^\times(b,n,\Lambda) = \w(b,n,\Lambda)$.
 Finally, we denote by $\w_v^\times(b,n,\Lambda)$ the union of $\w_{v,l}^\times(b,n,\Lambda)$ over all non negative integers $l$ with $l\mik \frac{b^n-1}{b-1}$.
 Observe that for every $f$ in $\w_v^\times(b,n,\Lambda)$, we have that either $f$ belongs to $\w_{v,0}^\times(b,n,\Lambda)$,
 or there exists unique positive integer $l$, which we denote by $\mathrm{sz}(f)$, and unique skew subtree $T$ of
 $b^{<n}$ with $l$ elements,
 which we denote by $\ws(f)$ such that $f$ is a $T$-variable word.
 In an identical way as in the case of $T$ being complete skew subtree, we define substitutions and spans (which we denote again by $[f]_\Lambda$) for the elements of $\w^\times_v(b,n,\Lambda)$ (if $f$ belongs to $\w_{v,0}^\times(b,n,\Lambda)$, then we set $[f]_\Lambda = \{f\}$).

\begin{defn}
  Let $b$ and $n$ be positive integers. A skew subtree $S$ of $b^{<n}$ is called semi-complete if every non-maximal node with respect to $\sqsubseteq$ has exactly $b$ immediate successors in $S$. Moreover, for a semi-complete skew subtree $S$, we define its interior $\mathrm{Int}(S)$ to be the set of all non-maximal nodes in $S$ with respect to $\sqsubseteq$.
\end{defn}
We have the following easy to observe remark.
\begin{rem}
  Let $b,n$ be positive integers and let $S$ be a semi-complete subtree of $b^{<n}$. Notice that $\mathrm{Int}(S)$ is
  empty if and only if $S$ is a singleton. Moreover, notice that if $\mathrm{Int}(S)$ is non empty then $\mathrm{Int}(S)$
  is a skew subtree.
\end{rem}

Let $b,n$ be positive integers and let $l$ be a non negative integer. We denote by $\w_{v,l}^*(b,n,\Lambda)$ the set of all pairs $(S,g)$ such that
\begin{enumerate}
  \item[(i)] $S$ is a semi-complete skew subtree of $b^{<n}$ having an interior with $l$ elements,
  \item[(ii)] $g$ belongs to $\w^\times_{v,l}(b,n,\Lambda)$ and
  \item[(iii)] if $l$ is positive then $\ws(g) = \mathrm{Int}(S)$.
\end{enumerate}
Moreover, for an element $f$ of $\w_v(b,n,\Lambda)$
and an element $(S,g)$ of $\w_{v,l}^*(b,n,\Lambda)$, we write $(S,g)\leq f$ if $S$ is a subtree of $\ws(f)$ and $[g]_\Lambda \subseteq [f]_\Lambda$. We also set
\[\w_{v,l}^*(f) = \{(S,g) \in \w_{v,l}^*(b,n,\Lambda) : (S,g) \leq f\} .\]

\subsection{The main result} The main result is a Ramsey type statement for structures of the form $ \w_{v,l}^*(b,n,\Lambda)$. However, it does not with arbitrary coloring but with these described in the following
definition.

\begin{defn}
  Let $b,n$ be positive integers, let $l$ be a non negative integer and let $\Lambda$ be a finite alphabet.
  First, we define a subset $A$ of $b^{<n}$ as follows. If $l$ is positive, then
  denote by $T$ the subtree of $b^{<n}$ consisting of the first $l$ elements of $b^{<n}$ with respect to $\preccurlyeq$
  and set
  \[A = \bigcup_{t\in T}\mathrm{ImSucc}_{b^{<n}}(t)\setminus T.\]
  If $l=0$ then set $A = \{\emptyset\}$, where by $\emptyset$ here we denote the empty sequence, that is, the root of the
  tree $b^{<n}$.
  Finally, let $f$ in $\w_v(b,n,\Lambda)$.
  We say that a coloring $c$ of $\w_{v,l}^*(b,n,\Lambda)$ is simple in $\w_{v,l}^*(f)$ if there exist disjoint subsets $B_1$ and $B_2$ (possibly empty) of $A$
  with $A= B_1\cup B_2$ satisfying the following.
  For every $(S,g)$ and $(S',g')$ in $\w_{v,l}^*(f)$ such that
  \begin{enumerate}
    \item [(i)] $\mathrm{Int}(S) = \mathrm{Int}(S')$,
    \item [(ii)] $\mathrm{I}_{S}(t) = \mathrm{I}_{S'}(t)$ for every $t$ in $B_2$ and
    \item [(iii)] $g(s) = g'(s)$ for every $s$ in $b^{<n}\setminus \bigcup_{t\in B_2}\mathrm{Succ}\big{(}\mathrm{I}_{S}(t)\big{)}$,
  \end{enumerate}
   we have that
  \[c\big((S,g)\big) = c\big((S',g')\big).\]
  Moreover, we say that a coloring $c$ of $\w_{v,l}^*(b,n,\Lambda)$ is simple if $c$ is simple in $\w_{v,l}^*(f)$,
  where $f(t) = v_t$ for every $t$ in $b^{<n}$.
\end{defn}

The following is the main result of the paper.

\begin{thm}
  \label{point_subsets}
  For every choice of positive integers $b,\ell, m,r$ and a non negative integer $l$ there exists a positive integer $n_0$ with the following property. For every integer $n$ with $n\meg n_0$ and every simple
  $r$-coloring of $\w_{v,l}^*(b,n,\Lambda)$ there exists $f$ in $\w_{v,m}(b,n,\Lambda)$ such that the set $\w_{v,l}^*(f)$ is monochromatic.
  We denote the least such $n_0$ by $\mathrm{PTGR}(l,m,b,r)$.
\end{thm}
Clearly Theorems \ref{tree_GR} and \ref{point_subsets} are equivalent.
In particular, for every choice of positive integers $k,m,b,r$ and a non negative integer $l$, we have that
\[\mathrm{TGR}(k,m,b,\ell,r)\mik\mathrm{PTGR}\Big(\frac{b^k-1}{b-1},m+1,b,r\Big)\]
and
\[\mathrm{PTGR}(l,m,b,r)\mik\mathrm{TGR}(k_l+1,m,b,\ell,r),\]
where $k_l$ is the smallest integer $k'$ satisfying $\frac{b^{k'}-1}{b-1}\meg l$.

\subsection{Proof of Theorem \ref{point_subsets}}

The proof of Theorem \ref{point_subsets} follows by induction on $l$.
The base case for $l=0$ follows trivially by Corollary \ref{tree_HJ_cor}.

\begin{defn}
  \label{defn_sign}
  Let $b,n,l$ be positive integers, let $\Lambda$ be a finite alphabet and let $(S,g)$ be an element of
  $\w_{v,l}^*(b,n,\Lambda)$.
  Denote by $S^o$ the set of all non-maximal with respect to $\preccurlyeq$ nodes of $\mathrm{Int}(S)$.
  Set $s_*=\mathrm{max}_\preccurlyeq(\mathrm{Int}(S))$ and
  \[D=\{t\in b^{<n}:t\preccurlyeq s_*\;\text{and}\;t\neq s_*\}.\]
  Moreover, we set $g' = g \upharpoonright D$, that is the restriction of $g$ on $D$ and
  \[S'=\mathrm{Int}(S)\cup\{\mathrm{min}_\sqsubseteq((\pred_{b^{<n}}(t)\cup\{t\})\setminus D):t\in S\setminus\mathrm{Int}(S)\;\text{and}\;s_*\not\sqsubseteq t\}.\]
We call the pair $(S',g')$, the signature of $(S,g)$ and we denote it by $\mathrm{sg}(S,g)$.

We say that a pair $\mathcal{S}=(S',g')$ is an $l$-signature if there exists an element $(S,g)$ of $\w_{v,l}^*(b,n,\Lambda)$ such that $\mathrm{sg}(S,g)=\mathcal{S}$. If we have in addition that $(S,g)$ belongs to
$\w_{v,l}^*(f,\Lambda)$, where $f$ belongs to $\w_v(b,n,\Lambda)$, we say that $\mathcal{S}$ is observable by $f$ (at $s_*$, where $s_*=\max_\preccurlyeq\mathrm{Int}(S)$).

Moreover, for every $f$ in $\w_v(b,n,\Lambda)$  and every $l$-signature $\mathcal{S}$ observable by $f$, we set
\[
  [\mathcal{S},f]=\{(\tilde{S},\tilde{g}) \in \w_{v,l}^*(f,\Lambda):
  \mathrm{sg}(\tilde{S},\tilde{g})=\mathcal{S}\}.
  \]
Finally, let $m$ be a positive integer, let $f$ be an element of $\w_v(b,n,\Lambda)$ and let $t_0$ in $\ws(f)$. We denote by $\langle t_0, f \rangle_m$ the set of all $f'$ in $\w_{v,m}(f,\Lambda)$
such that $t$ belongs to $\ws(f')$ for all $t$ in $\ws(f)$ with $t\preccurlyeq t_0$.
\end{defn}

Let us observe that
if $f'$ belongs to $\langle t_0, f\rangle_m$, where $m,t_0$ and $f$ are as in Definition \ref{defn_sign}, then we have that $f \upharpoonright \{s\in b^{<n}:s\preccurlyeq t_0\}
= f' \upharpoonright \{s\in b^{<n}:s\preccurlyeq t_0\}$ and, more importantly, that the set of signatures observable by $f$ at $t_0$ coincides with the set of signatures observable by $f'$ at $t_0$.

\begin{lem}
  \label{lem_ind_step}
  Let $b,\ell,m,l,r$ be positive integers and let $m'$ be a non negative integer with $m'< m-1$.
  Then there exists a positive integer $n_0$ with the following property.
  For every finite alphabet $\Lambda$ with at most $\ell$ elements and every
  choice of $n,N,f,c,t_0$ and $\mathcal{S}$ satisfying
  \begin{enumerate}
    \item[(i)] $n$ and $N$ are integers with $N\meg n\meg n_0$,
    \item[(ii)] $f$ belongs to $\w_{v,n}(b,N,\Lambda)$,
    \item[(iii)] $t_0$ belongs to $\ws(f)(m')$,
    \item[(iv)] $\mathcal{S}$ is an $l$-signature observable by $f$ at $t_0$ and
    \item[(v)] $c$ is a simple $r$-coloring of $\w_{v,l}^*(b,N,\Lambda)$,
  \end{enumerate}
  there exists $f'$ in $\langle t_0, f \rangle_m$ such that the set $[\mathcal{S}, f']$ is monochromatic. We denote the least such $n_0$ by $h_3(l,m',m,b,\ell,r)$.
\end{lem}
\begin{proof}
  Set $\ell'=\ell+m'+1$. We will show that
  \[h_3(l,m',m,b,\ell,r)\mik m'+\max_{d_0+d_2\mik b^{m'+1}}\mathrm{MTHJ}(d_0,0,d_2,b,\ell',\mathrm{CT}(1,m-m',b,d_2,r),r).\]
  Indeed, let $n\meg m'+\max_{d_0+d_2\mik b^{m'+1}}\mathrm{MTHJ}(d_0,0,d_2,b,\ell',\mathrm{CT}(1,m-m',b,d_2,r),r)$ and $n'=n-m'$. Also let $n,N,f,c,t_0,\Lambda$ and $\mathcal{S}=(S',g')$ be as in the statement of the lemma. 

  We set $T = \ws(f)$.
  Let $d$ be the cardinality of the set consisting of the $\sqsubseteq$-minimal points of $T\setminus\{t\in T : t \preccurlyeq t_0\}$ and let $(t_i)_{i\in[d]}$ be an enumeration of the elements of this set in $\preccurlyeq$-increasing order. Let $d_2$ be the cardinality of the set $R=(S'\setminus (\mathrm{Int}(S')\cup\{t_0\}))\cup \mathrm{ImSucc}_{b^{<n}}(t_0)$.
  Observe that for every $s$ in $R$ there exists unique $i$ in $[d]$ such that $s\sqsubseteq t_i$ and for every $i$ in
  $[d]$ there exists at most one $s$ in $R$ such that $s\sqsubseteq t_i$. We set $D_2$ to be the set of all $i$ in $[d]$ such that there exists some $s$ in $R$ satisfying $s\sqsubseteq t_i$. We also set $d_0=d-d_2$, $D_0=[d]\setminus D_2$, $\mathbf{D}=(D_0,\emptyset,D_2)$ and $\Lambda'=\{0,...,m'\}\cup \Lambda$, where we may assume without loss of generality that $\{0,...,m'\}$ and $\Lambda$ are disjoint. Finally, we fix some $a$ in $\Lambda$.

  For each $i$ in $[d]$, we define $P_i:b^{<n'}\to T$ by the rule
  \[P_i(z)=I_T(I_T^{-1}(t_i)^\con z)\]
  for all $z$ in $b^{<n'}$.
  We also define a map
  \[\mathcal{Q}:\w^\mathbf{D}_*(b,n',\Lambda)\to[\mathcal{S}, f]\]
  as follows. Let $\mathbf{w}=((w_i)_{i\in[d]},\emptyset,(s_i)_{i\in D_2})$ be an element of $\w^\mathbf{D}_*(b,n',\Lambda)$. We define
  $\mathcal{Q}(\mathbf{w})=(S_\mathbf{w},g_\mathbf{w})$
  as follows. We set \[S_\mathbf{w}=\mathrm{Int}(S')\cup\{t_0\}\cup\{P_i(s_i):i\in D_2\}\]
  and observe that $\mathrm{int}(S_{\mathbf{w}})=\mathrm{Int}(S')\cup\{t_0\}$.
  Moreover, for every $i$ in $[d]$, if the set $\mathrm{Pred}_{b^<N}(t_i)\cap \mathrm{Int}(S_{\mathbf{w}})$ is non empty,
  then we pick an onto map $\mathcal{R}_i:\{0,...,m'\}\to\{v_t:t\in \mathrm{Int}(S_{\mathbf{w}}) \}$, and otherwise, we set
  $\mathrm{R}_i:\{0,...,m'\}\to\{a\}$.
  Finally, we define $g_\mathbf{w}$ as follows.
  \begin{enumerate}
    \item[(i)] For every  $s$ in $b^{<N}$ such that $f(s)\in\Lambda$, we set $g_\mathbf{w}(s) = f(s)$.
    \item[(ii)] For every $s$ in $b^{<N}$ and $t$ in $T$, such that $f(s) = v_t$, $t\preccurlyeq t_0$ and $t \neq t_0$, we set $g_\mathbf{w} (s)= g'(t)$.
    \item[(iii)] For every $s$ in $b^{<N}$ such that $f(s) = v_{t_0}$, we set $g_\mathbf{w}(s) = f(s)$.

    \item[(iv)] For every $s$ in $b^{<N}$, $i$ in $[d]$ and $z$ in $b^{<n'}$, such that $f(s) = v_{P_i(z)}$, we set $g_\mathbf{w} (s)= w_i(z)$ if $w_i(z)\in \Lambda$ and $g_\mathbf{w} (s)= \mathrm{R}(w_i(z))$ if $w_i(z)\not\in \Lambda$.

    \item[(v)] For every $s$ in $b^{<N}$ such that there is no $i$ in $[d]$ and $z\in b^{<n'}$ such that
        $f(s)\not\in\Lambda$ and $f(s) = v_{P_i(z)}$ we set $g_\mathbf{w} (s)=a$.
  \end{enumerate}
  Clearly, the map $\mathcal{Q}$ is well defined. We define a coloring $c'$ on $\w^\mathbf{D}_*(b,n',\Lambda)$ setting
  \[c'(\mathbf{w})=c(\mathcal{Q}(\mathbf{w})).\]
  Since $c$ is simple we have that $c'$ is smooth.
  By Corollary \ref{tree_HJ_cor}, there exists $\big((f_i)_{i=1}^d,\emptyset\big)$ in
  $\w_v^\mathbf{D}(b,n',\Lambda)$ generating an $(m-m')$-dimensional combinatorial subspace such that the set
  $[\big((f_i)_{i=1}^d,\emptyset\big)]^*_\Lambda$ is monochromatic.

  We set
  \[T'=\{t\in T:\;t\preccurlyeq t_0\}\cup\bigcup_{i\in[d]}\{P_i(z):z\in\ws(f_i)\;\text{and}\;
  \mathrm{h}_{\ws(f_i)}(z)+\mathrm{h}_T(t_i)< m\}.\]
  Then $T'$ is a complete skew subtree of height $m$. We define $f'$ in $\mathrm{W}_{v,m}(b,n,\Lambda)$ as follows.
  \begin{enumerate}
    \item [(i)] For every  $s$ in $b^{<N}$ such that $f(s)\in\Lambda$, we set $f'(s) = f(s)$.
    \item [(ii)] For every $s$ in $b^{<N}$ and $t$ in $T$, such that $f(s) = v_t$ and $t\preccurlyeq t_0$, we set
                 $f'(s) = f(s)$.
    \item [(iii)] For every $s$ in $b^{<N}$, $i$ in $[d]$ and $z$ in $b^{<n'}$, such that
                $f(s) = v_{P_i(z)}$ and $f_i(z)\in \Lambda$, we set $f'(s) = f_i(z)$.
    \item [(iv)] For every $s$ in $b^{<N}$, $i$ in $[d]$, $z$ in $b^{<n'}$ and $y$ in $b^{<m-m'}$, such that
                $f(s) = v_{P_i(z)}$, $f_i(z) = v_y$ and $P_i(z)\in T'$, we set $f'(s) = v_t$ where $t=  \mathrm{I}_{T'}(\mathrm{I}_T^{-1}(t_i)^\smallfrown y)$.
    \item [(v)] For every $s$ in $b^{<N}$, $i$ in $[d]$, $z$ in $b^{<n'}$ and $y$ in $b^{<m-m'}$, such that
                $f(s) = v_{P_i(z)}$, $f_i(z) = v_y$ and $P_i(z)\not\in T'$, we set $f'(s) = a$.
    \item [(vi)] For every $s$ in $b^{<N}$ such that there is no $i$ in $[d]$ and $z\in b^{<n'}$ such that
        $f(s)\not\in\Lambda$ and $f'(s) = v_{P_i(z)}$ we set $f'(s) = a$.
  \end{enumerate}
  Notice that $\mathrm{ws}(f') = T'$ and $f'\in \langle t_0 , f \rangle_m$. Moreover, notice that
  $[\mathcal{S}, F']$ is a subset of $\{\mathcal{Q}(\mathbf{w}): \mathbf{w}\in [\big((f_i)_{i=1}^d,\emptyset\big)]^*_\Lambda\}$. Hence the set $[\mathcal{S}, F']$ is monochromatic and the proof is complete.
\end{proof}

Iterating the above lemma we obtain the following.

\begin{lem}
  \label{lem_ind}
  Let $b,\ell,m,l,r$ be positive integers.
  Then there exists a positive integer $n_0$ with the following property.
  For every finite alphabet $\Lambda$ with at most $\ell$ elements and every pair of integers
  $n$ and $N$ with $N\meg n\meg n_0$, every a simple $r$-coloring $c$ of $\w_{v,l}^*(b,N,\Lambda)$
  and every $f$ in $\w_{v,n}(b,N,\Lambda)$, there exists $f'$ in $\mathrm{W}_{v,m}(b,N,\Lambda)$ with
  $[f']_\Lambda \subseteq [f]_\Lambda$ such that for every $(S,g)$ and $(S',g')$ in
  $\mathrm{W}_{v,l}^*(f')$ with $\mathrm{sg}(S,g) = \mathrm{sg}(S',g')$ we have that
  $c((S,g)) = c((S',g'))$. We denote the least such $n_0$ by $h_4(l,m,b,\ell,r)$.
\end{lem}

\begin{proof}
  To provide an explicit upper bound for $h_4(l,m,b,\ell,r)$ we need to define some invariants.
  Let $R$ be the number of nodes $t$ in $b^{<m}$ such that there exists an $l$-signature observable by $b^{<m}$ at $t$ and let $(t_i)_{i=1}^R$ be an enumeration of these nodes in $\preccurlyeq$-increasing order. Moreover, for every $i$ in $[R]$ we set $\mathcal{Q}_i$ to be the set of all $l$-signatures observable by $b^{<m}$ at $t_i$.
  We set $P=\sum_{i=1}^R|\mathcal{Q}_i|$. Let
\[\iota_1:[P]\to\bigcup_{i=1}^R\{i\}\times\big[|\mathcal{Q}_i|\big]\]
be the unique bijection such that for every $p'< p$ in $[P]$ we have that either $i'<i$, or $i'=i$ and $q'<q$, where $(i',q')=\iota_1(p')$ and $(i,q)=\iota_1(p)$. Moreover, for every $p$ in $[P]$, we set $i_p$ and $q_p$ to be the unique integers satisfying $\iota_1(p)=(i_p,q_p)$.
Finally, we pick a bijection $\iota_2$ from $[P]$ onto the set of all signatures observable by $b^{<m}$, such that for every $p$ in $[P]$, we have that $\iota_2(p)$ belongs to $\mathcal{Q}_{i_p}$.

  By inverse recursion we define a sequence of integers $(M_p)_{p=0}^P$ by the rule
  \[
  \left\{ \begin{array} {l} M_P = m,\\
M_p=h_3(l,|t_{i_{p+1}}|,M_{p+1},b,\ell,r).
\end{array}  \right.
\]
We will show that $h_4(l,m,b,\ell,r)\mik M_0$. Indeed, let $n$ and $N$ be integers with $N\meg n\meg M_0$, let $c$ be a simple $r$-coloring of $\w_{v,l}^*(b,N,\Lambda)$ and let $f$ be in $\w_{v,n}(b,N,\Lambda)$.

For every $\varphi$ in $\mathrm{W}_{v}(b,N,\Lambda)$ with $\mathrm{h}_{\mathrm{ws}(\varphi)}\meg m$
and every $l$-signature $\mathcal{S}$
observable by $b^{<m}$ we define an $l$-signature
$\mathrm{Sg}(\mathcal{S},\varphi)$ observable by $\varphi$ as follows. Pick $(S_0,g_0)$ in $\mathrm{W}^*_{v,l}(b,m,\Lambda)$
such that $\mathrm{sg}(S_0,g_0) = \mathcal{S}$. Pick any $(S,g)$ in $\mathrm{W}^*_{v,l}(\varphi)$ satisfying the following.
\begin{enumerate}
  \item [(i)] $S= \{\mathrm{I}_{\mathrm{ws}(\varphi)}(s):s\in S_0\}$
  \item [(ii)] For every $t$ in $b^{<N}$ and $s$ in $b^{<m}$ with $s \preccurlyeq s_*$ and $s\neq s_*$, where $s_*=\mathrm{max}_\preccurlyeq(\mathrm{Int}(S))$, such that $\varphi(t) = v_s$, we have that $g(t) = g_0(s)$.
\end{enumerate}
Set $\mathrm{Sg}(\mathcal{S},\varphi) = \mathrm{sg}(S,g)$. Notice that $\mathrm{Sg}(\mathcal{S},\varphi)$ is well defined,
i.e. independent from the particular choice of $(S,g)$ in $\mathrm{W}^*_{v,l}(\varphi)$ satisfying (i) and (ii) above,
and observable by $\varphi$ at $s_*$ which belongs to $\bigcup_{i=0}^{m-2}\mathrm{ws}(\varphi)(i)$. Conversely, also notice that  for every
$\varphi$ in $\mathrm{W}_{v}(b,N,\Lambda)$ with $\mathrm{h}_{\mathrm{ws}(\varphi)}\meg m$, and every $l$-signature $\mathcal{S}'$ observable by $\varphi$  at some $t_0$ in $\bigcup_{i=0}^{m-2}\mathrm{ws}(\varphi)(i)$, there exists an
$l$-signature $\mathcal{S}$ observable by $b^{<m}$ such that $\mathrm{Sg}(\mathcal{S},\varphi) = \mathcal{S}'$.

Set $f_0 = f$ . We define $(f_p)_{p=0}^P$ inductively  satisfying the following for every $p=1,...,P$.
\begin{enumerate}
  \item [(i)] We have that $f_p$ belongs to $\langle t_{i_p},f_{p-1}\rangle_{M_p}$.
  \item [(ii)] The set $[\iota_2(p), f_p]$ is monochromatic.
\end{enumerate}
The inductive construction is straightforward by the definition of the numbers $(M_p)_{p=0}^P$
applying Lemma \ref{lem_ind_step}. Clearly, setting $f' = f_P$, we have that $f'$ satisfies the conclusion of the lemma and its proof is complete.
\end{proof}

We are ready to proceed with the proof of Theorem \ref{point_subsets}.

\begin{proof}
  [Proof of Theorem \ref{point_subsets}]
  As we have already mentioned the proof of Theorem \ref{point_subsets} follows by induction on $l$. The base case, when $l=0$ follows by Corollary \ref{tree_HJ_cor}.
  To complete the proof, we fix positive integer $l$ and we assume that the statement holds true for $l-1$.
  We also fix positive integers $m,b$ and $r$. We will show that
  \[\mathrm{PTGR}(l,m,b,r)\mik
  h_4(l,\mathrm{PTGR}(l-1,m,b,r)+1,b,r).\]
  Indeed, set $M_1=\mathrm{PTGR}(l-1,m,b,r)$ and $M_2=h_4(l,M_1+1,b,r)$ and pick any integer $n$ with $n\meg M_2$. By Lemma \ref{lem_ind},
  there exists $f_1$ in $\mathrm{W}_{v,M_1+1}(b,n,\Lambda)$
  such that every $(S,g)$ and $(S',g')$ in
  $\mathrm{W}_{v,l}^*(f_1)$ with the same signature have the same color. Let $\tilde{c}$ be the induced coloring
  on the $l$-signatures observable by $f_1$. Pick any $\alpha$ in $\Lambda$ and let $f_2$ be the element of $\mathrm{W}_{v,M_1}(b,n,\Lambda)$ such that $[f_2]_\Lambda \subseteq [f_1]_\Lambda$ and for every $t\in b^{<n}$ and
  $s\in\mathrm{ws}(f_1)(M_1)$ with $f_1(t) = v _s$ we have that $f_2(t)=\alpha$. Actually, the latter step of passing from
  $f_1$ to $f_2$ is necessary only in the case that $|S(\mathrm{h}_S-1)|=1$, where $S$ consists of the first $l$ elements of $b^{<n}$ with respect to $\preccurlyeq$, in order for the map $\mathcal{Q}$ introduced below to be well defined.

  Next we transfer $\tilde{c}$ to $\w_{v,l-1}^*(f_2)$. More precisely, we define a map $\mathcal{Q}$ from $\w_{v,l-1}^*(f_2)$ to the set of $l$-signatures observable by $f_1$ as follows.
  Fix an element $(\tilde{S}, \tilde{g})$  of $\w_{v,l-1}^*(f_1)$. Set  $s_*=\mathrm{min}_\preccurlyeq(\tilde{S}\setminus\mathrm{Int}(\tilde{S}))$ and
  \[D=\{t\in b^{<n}:t\preccurlyeq s_* \text{ and }t\neq s_*\}.\]
  Moreover, set $g = \tilde{g} \upharpoonright D$, that is the restriction of $\tilde{g}$ on $D$, and
  \[S
  =\mathrm{Int}(\tilde{S})\cup\{s_*\}\cup\{\mathrm{min}_\sqsubseteq((\pred_{b^{<n}}(t)\cup\{t\})\setminus D):t\in \tilde{S}\setminus\mathrm{Int}(\tilde{S})
  \}.\]
  Finally, set $\mathcal{Q}((\tilde{S}, \tilde{g}))=(S,g)$.
  It is easy to see that the image of $\mathcal{Q}$ is onto the $l$-signatures observable by $f_1$ and therefore covers
  the set of $l$-signatures observable by $f_2$. Moreover, observe that for every $f$ in
  $\mathrm{W}_{v}(b,n,\Lambda)$ such that $[f]_\Lambda \subseteq [f_2]_\Lambda$ we gave that
  the image of $\w_{v,l-1}^*(f)$ through $\mathcal{Q}$ covers the set $l$-signatures observable by $f$.
  We define an $r$-coloring $c'$ on $\w_{v,l-1}^*(f_2)$ by the rule $c'(\mathcal{S})=\tilde{c}(\mathcal{Q}(\mathcal{S}))$
  for every $\mathcal{S}$ in $\w_{v,l-1}^*(f_2)$.

  By the choice of $M_1$, making use of the inductive assumption, we obtain $f$ in $\w_{v,m}(b,n,\Lambda)$
  with $[f]_\Lambda \subseteq [f_2]_\Lambda$ such that the set $\w_{v,l-1}^*(f)$ is $c'$-monochromatic.
  By the definition of the color $c'$ it follows easily that $f$ satisfies the conclusion of Theorem \ref{point_subsets}.
\end{proof}

\section{Generalization to product of words} Actually the arguments in this paper can easily be modified and yield a more general form of Theorem \ref{tree_GR}. To state it we need first to introduce some pieces of notation.
Let $b,m,n$ be positive integers and $\Lambda$ a finite alphabet. We denote by $\w^d_{v,m}(b,n,\Lambda)$ the set of all variable words of $\w^d(b,n,\Lambda)$ that generate a combinatorial subspace of dimension $m$.
Moreover, if $f$ is a variable word of $\w^d(b,n,\Lambda)$, then we denote by $\w^d_{v,m}(f)$ the set of all elements $g$ in
$\w^d_{v,m}(b,n,\Lambda)$ such that $[g]_{\Lambda}\subseteq[f]_{\Lambda}$.

\begin{thm}
  \label{product_tree_GR}
  For every choice of positive integers $d,b,\ell, k,m$ and $r$ there exists a positive integer $n_0$ with the following property. For every integer $n$ with $n\meg n_0$, every finite alphabet $\Lambda$ with $\ell$ elements and every $r$-coloring of $\w^d_{v,k}(b,n,\Lambda)$ there exists $f$ in $\w^d_{v,m}(b,n,\Lambda)$ such that the set $\w^d_{v,k}(f)$ is monochromatic.
\end{thm}

\section{Appendix: Proof of Theorem \ref{Hales_Jewett*}}
This appendix is devoted to the proof of Theorem \ref{Hales_Jewett*}. Although the proof is a straightforward modification of Shelah's proof \cite{Sh} for the Hales--Jewett Theorem, we include the argument for the sake of completeness. Let us start with the following variation of Shelah's notion of insensitivity.

\begin{defn}\label{defn:*insensitivity}
  Let $\kappa,m,N$ be positive integers with $\kappa\mik m\mik N$ and $\Lambda$ a finite alphabet. Also let $w(v_0,...,v_{m-1})$ be an $m$-dimensional variable word over $\Lambda$ of length $N$ and $X$ the $\kappa^*$combinatorial subspace generated by $w$. Finally, let $L$ be a nonempty subset of $\Lambda$.
  We say that a finite coloring $c$ of $\Lambda^N\times\{0,...,N-1\}^{(\kappa)}$ is $L^*$insensitive in $X$ if
  \[c(w(a_0,...,a_{m-1}),\{\ell_i^w:i\in F\})=c(w(a_0',...,a_{m-1}'),\{\ell_i^w:i\in F\})\]
  for every choice of $F$ in $\{0,...,m-1\}^{(\kappa)}$ and $a_0,...,a_{m-1},a_0',...,a_{m-1}'$ in $\Lambda$
  satisfying the following:
  \begin{enumerate}
    \item[(i)] $a_i=a_i'$ for all $i$ in $F$,
    \item[(ii)] $a_j=a_j'$ for all $j$ in $\{0,...,m-1\}$ such that $a_j\in\Lambda\setminus L$ and
    \item[(iii)] $a_j=a_j'$ for all $j$ in $\{0,...,m-1\}$ such that $a_j'\in\Lambda\setminus L$.
  \end{enumerate}
\end{defn}
Let us first isolate the following observations on the notion of $^*$insensitivity.
\begin{rem}\label{rem_ins_sub}
  Let $\kappa,m,N$ and $\Lambda$ be as in Definition \ref{defn:*insensitivity}. Also, let $c$ be a coloring of
  $\Lambda^N\times\{0,...,N-1\}^{(\kappa)}$ and let $X$ be a $\kappa^*$combinatorial subspace. Then we have the following.
  \begin{enumerate}
    \item [(i)] If $L$ is a nonempty subset of $\Lambda$ and the coloring $c$ is $L^*$insensitive in $X$, then for every further $\kappa^*$combinatorial subspace $Y$ of $X$ (that is, $Y$ is a $\kappa^*$combinatorial subspace satisfying $Y\subseteq X$) we have that $c$ is $L^*$insensitive in $Y$.
    \item [(ii)] If $L_1$ and $L_2$ are nonempty subsets of $\Lambda$ such that the coloring $c$ is both $L_1^*$insensitive and $L_2^*$insensitive in $X$, then $c$ is $L_1\cup L_2^*$insensitive.
  \end{enumerate}
\end{rem}

Let us also introduce some pieces of notation.

\begin{notation}
  Let $N,M,m,\kappa$ be positive integers and $\Lambda$ be a finite alphabet. Also, let $X$ be an $M$-dimensional $\kappa^*$combinatorial subspace of $\Lambda^N\times\{0,...,N-1\}^{(\kappa)}$ and let $Y$ be an $m$-dimensional $\kappa^*$combinatorial subspace of $\Lambda^M\times\{0,...,M-1\}^{(\kappa)}$. We denote by $X[Y]$ the
  $m$-dimensional $\kappa^*$combinatorial subspace of $\Lambda^N\times\{0,...,N-1\}^{(\kappa)}$ generated by the
  $m$-dimensional variable word $w(v_0,...,v_{m-1})$ of length $N$ over $\Lambda$ defined as follows.
  Let $w_X(v_0,...,v_{M-1})$ be the $M$-dimensional variable word of length $N$ over $\Lambda$ that generates $X$ and
  let $w_Y(v_0,...,v_{m-1})$ be the $m$-dimensional variable word of length $M$ over $\Lambda$ that generates $Y$.
  Write $w_Y(v_0,...,v_{m-1}) = (z_i)_{i=0}^{M-1}$ and set $w(v_0,...,v_{m-1}) = w_X(z_0,...,z_{M-1})$.
\end{notation}

\begin{notation}
  Let $N,m,\kappa$ be a positive integers and let $\Lambda$ be a finite alphabet. Also, let $0\mik q_0<q_1<...<q_m\mik N$ be integers and let $X$ be an $m$-dimensional $\kappa^*$combinatorial subspace. We say that $X$ is $(q_i)_{i=0}^m$-compatible
  if for every $i=0,...,m-1$, we have that the wildcard set of $v_i$ in $w$ is contained in $[q_i,q_{i+1})$, where
   $w(v_0,...,v_{m-1})$ is the $m$-dimensional variable word that generates $X$.
\end{notation}
We isolate the following remark concerning composition of subspaces.
\begin{rem}\label{rem_compat_compos}
  Let $N,M,\kappa$ be positive integers and $\Lambda$ be a finite alphabet. Also, let $0\mik q_0<q_1<...<q_M\mik N$ and
  $0\mik p_0<p_1<...<p_m\mik M$ be integers.
  Finally, let $X$ be an $M$-dimensional $\kappa^*$combinatorial $(q_i)_{i=0}^M$-compatible subspace of $\Lambda^N\times\{0,...,N-1\}^{(\kappa)}$ and let $Y$ be an $m$-dimensional $\kappa^*$combinatorial $(p_i)_{i=0}^m$-compatible subspace of $\Lambda^M\times\{0,...,M-1\}^{(\kappa)}$. Then $X[Y]$ is an $m$-dimensional $\kappa^*$combinatorial $(q_{p_i})_{i=0}^m$-compatible subspace of $\Lambda^N\times\{0,...,N-1\}^{(\kappa)}$.
\end{rem}

We have the following variation of Shelah's insensitivity lemma.

\begin{lem}
  \label{Shelah_ins*}
  For every choice $k,\kappa,m,r$ of positive integers with $\kappa\mik m$  and $k\meg2$ there exist
  integers $n_0, q_0,...,q_m$ with $0=q_0<q_1<...<q_m=n_0$ satisfying the following. For every integer $N$ with $N\meg n_0$, every finite alphabet $\Lambda$ with $|\Lambda|=k$, every subset $L$ of $\Lambda$ with $|L|= 2$ and every $r$-coloring $c$ of $\Lambda^N\times\{0,...,N-1\}^{(\kappa)}$ there exists an $m$-dimensional $\kappa^*$combinatorial
  $(q_i)_{i=0}^m$-compatible subspace $X$ such that $c$ is $L^*$insensitive in $X$.
  We denote the least such $n_0$ by $\mathrm{Sh}^*(k,\kappa,m,r)$ and for every $j=0,...,m$ we denote the integer $q_j$ by
  $q_{\mathrm{Sh}^*}(j,k,\kappa,m,r)$.

  Moreover, the numbers $\mathrm{Sh}^*(k,\kappa,m,r)$ are upper bounded by a primitive recursive function belonging to the class $\mathcal{E}^4$ of Grzegorczyk's hierarchy.
\end{lem}
Before we proceed to the proof of Lemma \ref{Shelah_ins*} let us define a function $f_1:\nn^5\to\nn$ by the following rule. For every choice of positive integers $k,\kappa,m,r$ with $\kappa\mik m$ we recursively define
\[
\left\{ \begin{array} {l} f_1(k,\kappa,0,m,r)=0,\\
                               f_1(k,\kappa,i+1,m,r)=f_1(k,\kappa,i,m,r)+r^{(m-i-1+f_1(k,\kappa,i,m,r))^\kappa\cdot 
                               k^{m-i-1+f_1(k,\kappa,i,m,r)}}\end{array}  \right.
\]
and we set $f_1(k,\kappa,i,m,r)=0$ if at least one of the integers $k,\kappa,m,r$ is equal to zero or $m<\kappa$.
Observe that $f_1$ belongs to the class $\mathcal{E}^4$ of Grzegorczyk's hierarchy.

\begin{proof}
  [Proof of Lemma \ref{Shelah_ins*}]
  Let $k,\kappa,m,r$ be positive integers with $\kappa\mik m$.
  We will show the following inequality
  \begin{equation}
    \label{eq01}
    \mathrm{Sh}^*(k,\kappa,m,r)\mik f_1(k,\kappa,m,m,r).
  \end{equation}
  First, let us set $n_0=f_1(k,\kappa,m,m,r)$ and
  \begin{equation}
    \label{eq02}
    q_i=n_0-f_1(k,\kappa,m-i,m,r)
  \end{equation}
  for all $i=0,...,m$. Observe that $0=q_0<q_1<...<q_m=n_0$ as desired. Moreover, let us set
  \begin{equation}
    \label{eq03}
    p_i=q_i-q_{i-1}=r^{(i-1+f_1(k,\kappa,m-i,m,r))^\kappa\cdot k^{i-1+f_1(k,\kappa,m-i,m,r)}}
  \end{equation}
  $p_i=q_i-q_{i-1}$ for all $i=1,...,m$.

  Pick any integer $N$ with $N\meg n_0$, as well as, $\Lambda, L$ and $c$ as in the lemma.
  Let $a,b$ in $\Lambda$ satisfying $L = \{a,b\}$.
  We set $w_0(v_0,...,v_{n_0-1})=(v_0,...,v_{n_0-1})^\con\mathbf{x}$, where $\mathbf{x}$ is a constant word over $\Lambda$ of length $N-n_0$. Let $X_0$ be the $\kappa^*$combinatorial subspace generated by $w_0$. We inductively construct a decreasing sequence $(X_i)_{i=0}^m$ of $\kappa^*$combinatorial subspaces satisfying, for every $i=0,...,m$, the following.
  \begin{enumerate}
    \item[(i)] $X_i$ is of dimension $d_i=i+n_0-q_i$. Let $w_i(v_0,...,v_{d_i-1})$ be the $d_i$-dimensional variable word that generates $X_i$.
    \item[(ii)] If $i>0$, then $w_i\upharpoonright q_{i-1}=w_{i-1}\upharpoonright q_{i-1}$.
    \item[(iii)] If $i>0$, then we have that the wildcard set of $v_i$ in $w_i$ is contained in the interval $[q_{i-1},q_i)$.
    \item[(iv)] If $i>0$, then for every $F$ in $\{0,...,d_i-1\}^{(\kappa)}$ with $i\not\in F$ and every $a_0,...,a_{d_i-1}$ in $\Lambda$ with $a_i=a$, setting $a'_{j}=a_{j}$ for all $j\neq i$ and $a'_i=b$, we have that
        \[c(w_i(a_0,...,a_{d_i-1}),\{\ell_j^{w_i}:j\in F\})=c(w_i(a'_0,...,a'_{d_i-1}),\{\ell_j^{w_i}:j\in F\}).\]
  \end{enumerate}
  The inductive step of the construction is as follows. Assume that for some $i$ in $\{1,...,m\}$
  the subspaces $X_0,...,X_{i-1}$ have been constructed properly.
  By our inductive assumption (i), we have that $X_{i-1}$ is $d_{i-1}$-dimensional.
We set
   \[J=\{0,...,d_{i-1}-1\}\setminus[i-1,i-1+p_i).\]
   Also, we set $\mathcal{X}$ to be the set of all maps defined on $\Lambda^J\times J^{(\kappa)}$ and taking values in the set $\{1,...,r\}$. Let us observe that the cardinality of $J$ equals to $i-1+f_1(k,\kappa,m-i-1,m,r)$ and therefore
   \begin{equation}
     \label{eq04}
     |\mathcal{X}|\mik r^{(i-1+f_1(k,\kappa,m-i-1,m,r))^\kappa k^{i-1+f_1(k,\kappa,m-i-1,m,r)}}\stackrel{\eqref{eq03}}{=}p_i.
   \end{equation}

   We set $S$ to be the $d_{i-1}$-dimensional combinatorial subspace generated by $w_{i-1}$ and we define
   a map $Q:\{a,b\}^{p_i}\times\Lambda^{J}\to S$ as follows.
   For every $\mathbf{a}=(a_t)_{t=0}^{p_i-1}$ in $\{a,b\}^{p_i}$ and $\mathbf{b}=(b_s)_{s\in J}$ in $\Lambda^J$ we set
   \[Q(\mathbf{a},\mathbf{b})=w_{i-1}(b_0,...,b_{i-2},a_0,...,a_{p_i-1},b_{i-1+p_i},...,b_{d_{i-1}-1}).\]
   Moreover, for every $\mathbf{a}$ in $\{a,b\}^{p_i}$ we define $g_{\mathbf{a}}$ in $\mathcal{X}$ as follows. For every $(\mathbf{b},F)$ in $\Lambda^J\times J^{(\kappa)}$ we set $q_\mathbf{a}(\mathbf{b},F)=c(Q(\mathbf{a},\mathbf{b}),\{\ell_j^{w_{i-1}}:j\in F\})$.
   We consider the sequence $(\mathbf{a}_t)_{t=0}^{p_i}$ in $\{a,b\}^{p_i}$ defined by the rule
   \[
\mathbf{a}_t=(\underbrace{a,...,a}_{t-\mathrm{times}},  \underbrace{b,...,b}_{(p_i-t)-\mathrm{times}}).
\]
By \eqref{eq04} and the pigeonhole principle, there exist $s_1<s_2$ in $\{0,...,p_i\}$ such that $q_{\mathbf{a}_{s_1}}=q_{\mathbf{a}_{s_2}}$. Let $\mathbf{a}_{s_1}=(a^1_t)_{t=0}^{p_i-1}$ and
$\mathbf{a}_{s_2}=(a^2_t)_{t=0}^{p_i-1}$ and observe that there exists unique $d_i$-dimensional variable word $w_i$ of length $N$ over $\Lambda$ such that
\[\big\{w_i(b_0,...,b_{i-2},x,b_{i-1+p_i},...,b_{d_{i-1}-1}):(b_t)_{t\in J}\in \Lambda^J\;\text{and}\;x\in\{a,b\}\big\}\]
is equal to the set $\{Q(\mathbf{a}_{s_{t}},\mathbf{b}):t\in\{1,2\}\;\text{and}\;\mathbf{b}\in\Lambda^J\}$.
Setting $X_i$ to be the $\kappa^*$combinatorial subspace generated by $w_i$, it follows readily that $X_i$ is as desired and the inductive step of the construction is complete.

Set $X$ to be the $\kappa^*$combinatorial subspace $X_m$. It follows readily that $X$ is as desired. The proof of inequality \eqref{eq01} is complete. Since $f_1$ belongs to the class $\mathcal{E}^4$ of Grzegorczyk's hierarchy, the proof of the lemma is complete.
\end{proof}

Define a function $f_2:\nn^5\to\nn$ by the following rule. For every choice of positive integers $k,\kappa,m,r$ with $\kappa\mik m$ we recursively define
\[
\left\{ \begin{array} {l} f_2(0,k,\kappa,m,r)=0,\\
                            f_2(1,k,\kappa,m,r)=m,\\
                            f_2(i+1,k,\kappa,m,r)=f_1(k,\kappa,f_2(i,k,\kappa,m,r), f_2(i,k,\kappa,m,r),r)
                               \end{array}  \right.
\]
and we set $f_1(k,\kappa,i,m,r)=0$ if at least one of the integers $k,m,r$ is equal to zero or $m<\kappa$.
Observe that $f_2$ belongs to the class $\mathcal{E}^5$ of Grzegorczyk's hierarchy. Moreover, by \eqref{eq01}, for every choice of positive integers $i,k,\kappa,m,r$ with $\kappa\mik m$ we have that
\begin{equation}\label{eq13}
  \mathrm{SH}^*(k,\kappa,f_2(i,k,\kappa,m,r), r)\mik f_2(i+1,k,\kappa,m,r) .
\end{equation}
We are ready for the proof of Theorem \ref{Hales_Jewett*}.

\begin{proof}
  [Proof of Theorem \ref{Hales_Jewett*}]
  Let $k,\kappa,m$ and $r$ be positive integers with $\kappa\mik m$. For $k=1$  the statement of the theorem is trivial and
  for $k=2$, the result follows by Lemma \ref{Shelah_ins*}.
  Assume that $k\meg 3$. Set $n_0 =  f_2(k,k,\kappa,m,r)$.
  Moreover, we define the sequence $(q_i)_{i=0}^m$ as follows.
  For every $j=1,...,k$ set $m_j = f_2(k+1-j,k,\kappa,m,r)$. Then, by \eqref{eq01}, for every $j=2,...,k$, we have that
  \[0 = q_{\mathrm{Sh}^*}(0,k,\kappa,m_{j},r)
  <q_{\mathrm{Sh}^*}(1,k,\kappa,m_{j},r)<...<q_{\mathrm{Sh}^*}(m_{j},k,\kappa,m_{j},r)
  \mik m_{j-1}.\]
  By induction define for every $j=1,...,k$ a sequence of integers  $(q(i,j))_{i=0}^{m_j}$ as follows.
  Set $q(i,1)=i$ for all $i=0,...,m_1$ and for every $j=1,...,k-1$ set $q(i,j+1)
  = q( q_{\mathrm{Sh}^*}(i,k,\kappa,m_{j+1},r), j )$ for all $i=0,...,m_{j+1}$. Notice that $m_k = m$.
  Set $q_i = q(i, k)$ for all $i = 0,...,m$.

  To complete the proof we show that $n_0$ and $(q_i)_{i=0}^m$ satisfy the conclusion of the theorem.
  By the definition of these quantities, we have that $0 = q_0<q_1<...<q_m\mik n_0$.
  Let $N$ be an integer with $N\meg n_0$, let $\Lambda$ be a finite alphabet with $k$ elements and let $c$ be an $r$-coloring of $\Lambda^N\times\{0,...,N-1\}^{(\kappa)}$.
  Write $\Lambda = \{a_1,...,a_k\}$ and set $L_j = \{a_j,a_{j+1}\}$ for all $j=1,...,k-1$.
  Inductively we construct a decreasing sequence $(X_j)_{j=1}^k$ of $\kappa^*$combinatorial subspaces of
  $\Lambda^N\times\{0,...,N-1\}^{(\kappa)}$ such that for every $j=1,...,k$ the following are satisfied.
  \begin{enumerate}
    \item [(i)] The subspace $X_j$ is $m_j$-dimensional and $(q(i,j))_{i=0}^{m_j}$-compatible.
    \item [(ii)] If $j>1$ then the coloring $c$ is $L_{j-1}^*$insensitive in $X_j$.
  \end{enumerate}
  We set $w_1(v_0,...,v_{n_0-1})=(v_0,...,v_{n_0-1})^\con\mathbf{x}$, where $\mathbf{x}$ is a constant word over $\Lambda$ of length $N-n_0$. Let $X_1$ be the $\kappa^*$combinatorial subspace generated by $w_1$. Assume that for some $j\in\{1,...,k-1\}$ we have constructed the properly the subspaces $X_1,...,X_j$. We describe the construction of
  $X_{j+1}$. Let $w_j(v_0,...,v_{m_j-1})$ be the $m_j$-variable word generating $X_j$.
  Define an $r$-coloring $\tilde{c}$ of $\Lambda^{m_j}\times\{0,...,m_j-1\}^{(\kappa)}$
  setting for every $(\mathbf{a},F)$ in $\Lambda^{m_j}\times\{0,...,m_j-1\}^{(\kappa)}$
  \[\tilde{c}(\mathbf{a},F) = c(w_j(a_0,...,a_{m_j-1}),\{\ell^{w_j}_i: i\in F\}),\]
  where $\mathbf{a}= (a_i)_{i=0}^{m_j-1}$. By Lemma \ref{Shelah_ins*},
  there exists an $m_{j+1}$-dimensional $\kappa^*$combinatorial
  $(q_{\mathrm{Sh}^*}(i,k,\kappa,m_{j+1},r))_{i=0}^{m_{j+1}}$-compatible subspace $Y$ such that $\tilde{c}$ is $L_j^*$insensitive in $Y$. Set $X_{j+1} = X_j[Y]$. Then $X_{j+1}$ is an $m_{j+1}$-dimensional further $\kappa^*$subspace of $X_j$ and $c$ is
  $L_j^*$insensitive in $X_{j+1}$. Moreover, by Remark \ref{rem_compat_compos}, the fact that $X_j$ is $(q(i,j))_{i=0}^{m_j}$-compatible and the definition of the numbers $((q(i,j))_{i=0}^{m_j})_{j=1}^k$, we have that $X_{j+1}$ is $(q(i,j+1))_{i=0}^{m_{j+1}}$-compatible. The inductive construction of $(X_j)_{j=1}^k$ is complete.

  Set $X = X_k$. Notice that $m_k = m$. Thus, invoking the definition of the sequence $(q_i)_{i=0}^m$, we have that
  $X$ is an $m$-dimensional $\kappa^*$combinatorial $(q_i)_{i=0}^m$-compatible subspace of $\Lambda^N\times\{0,...,N-1\}^{(\kappa)}$. Moreover, since $X$ is a subspace of $X_j$ for every $j=2,...,k$, by condition (ii) of the inductive construction of $(X_j)_{j=1}^k$ and Remark \ref{rem_ins_sub}, we have that c is strongly $^*$insensitive in $X$.
\end{proof}

\section*{Acknowledgment}

The research was supported by the Hellenic Foundation for Research and
Innovation (H.F.R.I.) under the ``2nd Call for H.F.R.I. Research Projects
to support Faculty Members \& Researchers'' (Project Number: HFRI-FM20-02717).

The research on this paper is partially supported by
grants from NSERC(455916) and CNRS(UMR7586).


\end{document}